\documentclass{article}
\usepackage{bm,amssymb,amsmath,graphicx,amsthm,wrapfig,caption,subcaption,epstopdf,esint,cancel,xcolor}
\usepackage[DIV=15]{typearea}
\date{}

\newtheorem{remark}{Remark}
\newtheorem{lemma}{Lemma}

\newtheorem{theorem}{Theorem}

\let\Delta\relax
\newcommand{\Delta}{\nabla^2}

\begin{document}


\title{Poisson-Nernst-Planck equations with steric effects - non-convexity and multiple stationary solutions}

 \author{Nir Gavish\thanks{ngavish@technion.ac.il}\thanks{Mathematics Department, Technion - Israel Institute of Technology, Haifa, 32000, Israel}}

\maketitle
\begin{abstract}
We study the existence and stability of stationary solutions of Poisson-Nernst-Planck equations with steric effects (PNP-steric equations)
with two counter-charged species.  We show that within a range of parameters, steric effects give rise to multiple solutions { of the corresponding stationary equation} that are smooth.  { The PNP-steric equation, however, is found to be ill-posed at the parameter regime where multiple solutions arise.}
Following these findings,  we introduce a novel PNP-Cahn-Hilliard model,  show that it { is well-posed and that it} admits multiple stationary solutions that are smooth and stable.  The various branches of stationary solutions and their stability are mapped utilizing bifurcation analysis and numerical continuation methods.  
\end{abstract}
\section{Introduction}
It is difficult to overstate the importance of electrolyte solutions in biological and electrochemical systems.  The pioneering studies of Nernst and Planck in 1889-1890~\cite{nernst1889elektromotorische,planck1890ueber} led to the Poisson-Nernst-Planck (PNP) theory, which is one of the most widely used continuum mean-field model to describe electrolyte behavior.  PNP theory, however, is valid only at exceedingly dilute solutions due to various underlying assumptions.  In particular, the model assumes point-like ions, while neglecting crowding effects due to the finite-size of the ions.  The shortcomings of the PNP model has led to the development of a wide family of generalized Poisson-Nernst-Planck models that take into account ion-ion and/or ion-solvent interactions~\cite{gillespie2002coupling,bikerman1942xxxix,borukhov1997steric,kilic2007steric,ben2009beyond,lopez2011poisson,stern1924theorie,di2004specific,di2003simple,mansoori1971equilibrium,Horng:2012io,gongadze2015asymmetric,ben2011ion,booth1951dielectric,hatlo2012electric,ben2011dielectric,psaltis2011comparing,bazant2011double,gavish2015systematic,liu2014poisson,eisenberg2010energy}; for an overview the reader is also referred to~\cite{Bazant:2009fp,iglivc2015nanostructures} and references within.  Among these models, we particularly focus here on the PNP-steric~\cite{lin2014new,Horng:2012io} model
\begin{subequations}\label{eq:stericPNP_intro}
\begin{equation}
\begin{split}
&\frac{\partial c_i}{\partial t}+\nabla\cdot J_i=0,\qquad i=1,2,\cdots,N,\\
&\phi_{xx}+\sum_{i=1}^N z_ic_i=0,
\end{split}
\end{equation}
where
\begin{equation}\label{eq:JiPNP}
J_i=-c_i\nabla \left(\log c_i+z_i\phi+\sum_{i=1}^N g_{ij}c_j\right).
\end{equation}
\end{subequations}
The above system is written using non-dimensional variables, where~$c_i$ are the concentrations of the~$i^{\rm th}$ ionic species with corresponding valence~$z_i$,~$\phi$ is the electric potential, and~$g_{ij}=g_{ji}$ are non-negative constants depending on ion radii and the energy coupling constant of the~$i^{\rm th}$ and~$j^{\rm th}$ species ions~\cite{Horng:2012io}.  When~$g_{ij}=0$ for all~$1\le i,j\le N$, the PNP-steric equations reduces to the classical PNP model that does not account for finite-size effects.
  
A wide family of generalized PNP models share common features - these models give rise to unique stationary solutions which in the absence of applied external potential reduces to a homogenous state~\cite{Bazant:2009fp,EJM:10386443}.  The uniqueness of a stationary solution, however, is challenged by observations of ionic currents through ionic channels.  Indeed, experiments~\cite{eisenberg1996atomic} and simulations~\cite{kaufman2013multi} measuring currents through single protein channels show that single channels switch between `open’ or `closed’ states in a spontaneous stochastic process called gating. Currents are either (nearly) zero or at a definite level, characteristic of each type of protein, independent of time, once the channel is open.  These observations of a bi-stable system strongly suggest that an underlying model for ion channels should have multiple stationary solutions, one corresponding to an `open' state and another to a `closed' state.  Following this insight, the recent work~\cite{lin2015multiple} due to Lin and Eisenberg considered the existence of multiple stationary solutions of the (non-dimensional) PNP-steric equations~\eqref{eq:stericPNP_intro} { in a general (namely, not in an ion-channel) setting}.  

{ We note that PNP equations with piecewise constant fixed charge can give rise to multiple solutions, see, e.g.,~\cite{rubinstein1987multiple,liu2009one,eisenberg2007poisson}.  The work~\cite{lin2015multiple}, as well as the current study, is concerned whether the PNP-steric equations with spatially constant permanent charges can give rise to multiple steady-states.  Namely, whether steric effects themselves are sufficient to give rise to multiple stationary solutions.}
The study~\cite{lin2015multiple} showed that when~$N=2$,~$z_1>0>z_2$ and~$g_{11}=g_{22}$, the PNP-steric system gives rise to a single stationary solution, but  proved the existence of multiple stationary solutions in multi-species systems,~$N\ge 3$.    This work, however, considered solely the stationary equations satisfying
\begin{equation}\label{eq:steadyStateStericPNP_Boltzmann}
\log c_i+z_i\phi+\sum_{i=1}^2 g_{ij}c_j=\lambda_i,\quad i=1,\cdots,N,
 \end{equation}
 where~$\lambda_i$ is a Lagrange multiplier associated with the conservation of mass, coupled with Poisson's equation.  Particularly, these studies characterized branches of solutions that are functions of~$\phi$, i.e.~$c_i=c_i(\phi)$, defined by the algebraic system~\eqref{eq:steadyStateStericPNP_Boltzmann}.  Dynamics were not considered in these studies.  It, therefore, remains an open question whether the PNP-steric model with two species admits smooth multiple stationary solutions, and whether these solutions would be stable under PNP-steric dynamics.  { These open questions are addressed in this study.}
 
 \subsection{Main results and paper outline}
In this work, we study the PNP-steric system with two counter-charged species.  We show that under certain conditions, particularly, for sufficiently concentrated (large~$L^1$ norm or average ionic concentration) initial conditions, the { reduced time-independent system for the stationary states} does give rise to smooth multiple  solutions.  { However, at the parameter regime where multiple stationary solutions arise, the PNP-steric equation is found to be ill-posed}.  The finding that the PNP-steric system { is ill-posed for sufficiently} concentrated solutions raises a need to develop a model that is valid in this regime.   In the second part of this study, we introduce a novel PNP-Cahn-Hilliard (PNP-CH) model which is { well-posed} also at high bulk concentrations.  We show that this model gives rise to multiple stationary solutions that are stable with respect to PNP-CH dynamics.  These stationary solutions differ by their bulk structure.  Utilizing bifurcation analysis and a numerical continuation study, we map the various branches of stationary solutions.   
  
The paper is organized as follows. In Section~\ref{sec:model} we briefly review the PNP-steric model derivation.  In Section~\ref{sec:convex} we show that under certain conditions, the corresponding free energy of the PNP-steric model is convex, and review the theory of generalized PNP models with convex energy.  In particular, these models give rise to a unique stationary solution.  Section~\ref{sec:nonConvex} focuses on the case where the corresponding free energy of the PNP-steric system is non-convex.   
In contrast to the previous works, e.g.~\cite{lin2015multiple}, we do not focus on the study of the solutions profiles~$c_i(\phi)$, but rather present a different approach, and consider the solutions trajectories in the~$(c_1,c_2)$ plane.  We show that the~$(c_1,c_2)$ plane can be divided into two regions,~$D^{\rm convex}$ in which the corresponding free energy of the PNP-steric system is locally convex and~$D^{\rm concave}$ in which the free energy is locally concave in~$(c_1,c_2)$.  We then present a  classification of trajectories based on their path within these two regions.
Solution trajectories which are fully contained in~$D^{\rm convex}$ (aka type I trajectories) are studied in Section~\ref{sec:typeI}, and are shown to reflect the case where steric interactions are of perturbative nature, namely, they do not change qualitative properties of the stationary solutions .  In contrast, solutions which at the bulk reside in~$D^{\rm concave}$ reflect a case where steric effects change the qualitative behavior of the equation and give rise to structure formation in the bulk (aka type III solutions).  In Section~\ref{sec:typeIII} we characterize periodic solutions of the PNP-steric system in this case, and show that they co-exist with the homogenous stationary solution.  The above, thus, presents a full characterization of multiple stationary solutions of the PNP-steric model { via the study of the reduced time-independent equations}.  
Next, we go beyond the study of the stationary equations and consider stability of the stationary solutions with respect to PNP-steric dynamics.  Our results show the PNP-steric system becomes { ill-posed} for concentrated solutions, and in particular in the regimes where multiple stationary solutions arise.  Section~\ref{sec:interimSummary} presents an interim summary of the study of the PNP-steric model, as well as a discussion regarding the reasons for the failure of the PNP-steric system to describe concentrated solutions:
Intuitively,  in classical PNP theory, it is the entropy that regularizes the solution and drives the system toward a spatially homogenous state.  At high bulk concentrations, we show, however, that steric effects dominate entropic effects.  Consequently, entropy fails to regularize the solution and the system becomes { ill-posed} as it is pushed toward higher and higher frequencies.  To overcome this problem, in Section~\ref{sec:PNPCH} we introduce a gradient energy term which regularizes the solution by direct energetic penalty of large gradients, giving rise to the PNP-Cahn-Hilliard (PNP-CH) system.  Such gradient energy terms arise naturally in Cahn-Hilliard (Ginzburg-Landau) mixture theory~\cite{cahn1958free}, and has been recently used to model steric interactions in ionic liquids~\cite{gavish2016theory}.   
Section~\ref{sec:PNPCH_derivation} presents the derivation of the PNP-CH system.  Sections~\ref{sec:linearAnalysis} and~\ref{sec:weaklyNonlinear} present the linear and weakly nonlinear stability analysis of the homogenous state in an infinite domain, and characterizes the regimes in which super-critical and sub-critical bifurcations arise.  Sections~\ref{sec:finiteDomain} and~\ref{sec:appliedVoltage} consider the case of finite domain with possible applied voltage using numerical continuation techniques.   The results present a wide family of stable stationary solutions of the PNP-CH system, showing a that finite domain has a significant effect on the solutions and their stability.  Numerical methods are discussed in Section~\ref{sec:numerics}, and concluding remarks are presented in  Section~\ref{sec:discussion}.
\section{The PNP-steric Model}\label{sec:model}
We briefly recall the PNP-steric model~\cite{lin2014new,Horng:2012io}.  For simplicity, we consider only the case of an electrolyte with two ionic species~$c_1,c_2$ with valences~$z_1, z_2$, respectively, such that
\begin{equation}\label{eq:valence}
z_2<0<z_1.
\end{equation}
The solution is bounded between two electrodes which are situated at~$x=-L$ and~$x=L$, with surface charge density such that the electric potential~$\phi$ on the walls equals,
\begin{equation}\label{eq:boundaryPhi}
\phi(\pm L)=\phi_{\pm L},
\end{equation}
where the reference potential is taken to be zero voltage at the bulk\footnote{Namely, at~$L\to\infty$, the average electric potential at the interior of the domain is zero.  This implies that the homogenous state~$(\bar{c}_1,\bar{c}_2)$ is a stationary solution of the PNP-steric system when~$\phi_{\pm L}=0$. }.  
The system is globally electroneutral,
\begin{equation}\label{eq:zeroVoltageRefPoint}
z_1\bar{c}_1+z_2 \bar{c}_2+\rho_0=0,
\end{equation}
where~$\bar{c}_i$ is the average density of the~$i^{\rm th}$ ionic species
\begin{equation} \label{eq:chargeConservation}
\bar{c}_i:=\frac1{2L}\int_{-L}^L c_i(x) \,dx,
\end{equation}
and where~$\rho_0$ is a possible permanent `background' charge density. The solvent is described as a dielectric medium with a uniform dielectric response~$\epsilon$.  

The action potential describing the system is given by
\begin{equation}
\label{eq:stericPNP_action_potential}
\mathcal{A}(c_1,c_2,\phi) = \int_{-L}^{L} k_BT\underbrace{\sum_{i=1}^2c_i\left(\ln \frac{c_i}{\bar{c}_i}-1\right)}_{\rm entropy}+\underbrace{q\,{\bf z}^T\,{\bf c}\phi  - \frac{\epsilon}{2}|\nabla \phi|^2 }_{\rm electrostatic}+\underbrace{\frac12{\bf c}^T {\bf G}{\bf c}}_{\rm steric}\, dx,\qquad 
\end{equation}
where~${\bf c}=(c_1,c_2)$,~${\bf z}=(z_1,z_2)$,~${\bf \bar c}=(\bar c_1,\bar c_2)$,~$q$ is the unit of electrostatic charge,~$k_B$ is Boltzmann's constant, and~$T$ is the temperature.  The action potential consists of the standard entropic and electrostatic free energy together with a correction functional~${\bf c}^T {\bf G}{\bf c}$, where~${\bf G}=[g_{ij}]$ is a symmetric $2\times 2$ matrix with non-negative entries, which accounts for steric interactions between ions~\cite{lin2014new,Horng:2012io}.  Particularly, when~${\bf G}\equiv {\bf 0}$, steric interactions are neglected and the action potential~$\mathcal{A}$ reduces to the action potential corresponding to the Poisson-Nernst-Planck model.

Requiring that~$\phi$ is a critical point of the action~$\mathcal{A}$ yields Poisson's equation
\begin{subequations}\label{eq:stericPNP_dimensional}
\begin{equation}\label{eq:genPNP_Poisson_dimensional}
 0 = \frac{\delta \mathcal{A}}{\delta \phi} =\varepsilon \phi_{xx}+q \left({\bf z}^T{\bf c}+\rho_0\right)=0.
\end{equation}
The (Wasserstein) gradient flow along the action,
\[
\frac{dc_i}{dt}=\breve{D}\nabla\cdot \left(c_i\nabla\frac{\delta  \mathcal{A}}{\delta c_i}\right),\quad i=1,2,
\]
where~$\breve{D}$ is the mobility, 
gives rise to generalized Nernst-Planck equations for the ionic transport
\begin{equation}\label{eq:Nernst-Planck}
\frac{dc_i}{dt}+\nabla \cdot J_i=0,\qquad J_i=-\breve{D}c_i\nabla\frac{\delta  \mathcal{A}}{\delta c_i}=-\breve{D}c_i\nabla \left(k_BT\log c_i+z_iq\phi+\sum_{j=1}^2 g_{ij}c_j\right).
\end{equation}
\end{subequations}
Introduction of the non-dimensional variables 
\begin{equation}\label{eq:scaling}
\tilde\phi=\frac{q}{k_BT}\phi,\quad \tilde x=\frac{x}{\lambda_D},\quad \tilde c_i=\frac{c_i}{c^o},\quad \tilde{\bar{c}}_i=\frac{\bar{c}_i}{c^o}, \quad \tilde{\bf G}=\frac1{k_BT}{\bf G},\quad \tilde t=\frac{\lambda_D^2}{\breve{D}\,k_BT}t,
\end{equation}
where~$\lambda_D$ is the Debye length 
\[
\lambda_D:=\sqrt{\frac{k_BT\varepsilon}{q^2c^o}},
\]
and~$c^o$ is an arbitrary reference density, see remark~\ref{rem:c0} below, yields  the PNP-steric equation (presented after omitting the tildes)
\begin{subequations}\label{eq:stericPNP}
\begin{equation}\label{eq:genPB_dimensional}
\begin{split}
&\frac{dc_i}{dt}+\nabla \cdot J_i=0,\qquad J_i=-c_i\nabla \left(\log c_i+z_i\phi+\sum_{j=1}^2 g_{ij}c_j\right),\\
&\phi_{xx}=-\left({\bf z}^T{\bf c}+\rho_0\right),
\end{split}
\end{equation}
subject to the boundary conditions
\begin{equation}\label{eq:genPNP_nondimensional_BC}
\phi(\pm L)=\phi_{\pm L},\quad \left. J_i\right|_{\pm L}=0.
\end{equation}
\end{subequations}
In particular, the stationary charge density profiles satisfy the Poisson-Boltzmann-steric (PB-steric) equations, composed of generalized Boltzmann equations
\begin{subequations}\label{eq:stericPB}
\begin{equation}\label{eq:genPB_Boltzmann_nondimensional}
\frac{\delta \mathcal{A}}{\delta c_i}=\log c_i+z_i\phi+\sum_{j=1}^2 g_{ij}c_j=\lambda_i(\bar{\bf c})=\log \bar c_i+\sum_{j=1}^2 g_{ij}\bar c_j,\quad i=1,2,\\
\end{equation}
coupled with Poisson's equation
\begin{equation}\label{eq:stericPB_Poisson}
\phi_{xx}=-\left({\bf z}^T{\bf c}+\rho_0\right),
\end{equation}
and subject to the boundary conditions
\begin{equation}\label{eq:genPB_nondimensional_BC}
\phi(\pm L)=\phi_{\pm L}.
\end{equation}
Here,~$\lambda_i$ is chosen so the reference electric potential is (arbitrarily) set to~$\phi=0$ when~$c_i=\bar{c}_i$.
\end{subequations}

\begin{remark}\label{rem:c0}Note the explicit dependence of~$\lambda_i$ in~${\bf \bar{c}}$, see~\eqref{eq:genPB_Boltzmann_nondimensional}.  Typically, this explicit dependence is eliminated by setting the reference density~$c^o$, see~\eqref{eq:scaling}, to be the bulk density, e.g.,~$\bar{c}=\bar{c}_1$.  
We retain an arbitrary reference density, $c^o$, in order to preserve the central role played by the connection between the bulk densities $\bar{c}_i$ and the Lagrange multipliers.
\end{remark}

\section{The case of convex energy and the functions $c_i(\phi)$}\label{sec:convex}
The free energy is attained by substituting Poisson's equation~\eqref{eq:genPNP_Poisson_dimensional} into the action potential~\eqref{eq:stericPNP_action_potential}, and rescaling by~\eqref{eq:scaling}, yielding
\begin{equation}
\label{eq:genPB_free_energy_with_L}
\mathcal{E}(c_1,c_2,\phi) = \int_{-L}^{L} \underbrace{h(c_1,c_2)}_{\rm entropy+steric}+\underbrace{\frac{\epsilon}{2}|\nabla \phi|^2 }_{\rm electrostatic}\, dx, 
\end{equation}
where
\begin{equation}\label{eq:h_of_stericPNP}
h({\bf c})=\sum_{i=1}^2c_i\left(\ln \frac{c_i}{\bar{c}_i}-1\right)+\frac12{\bf c}^T {\bf G}{\bf c}.
\end{equation}
In~\cite{EJM:10386443}, we have studied the case of a general function~$h$, aka the family of generalized PNP models with corresponding free energy~\eqref{eq:genPB_free_energy_with_L}.  In particular, we focused on the case when the function~$h$ is strictly convex.  As will be shown, in this case the corresponding generalized PNP model cannot give rise to multiple stationary solutions. Nevertheless, for completeness and to allow comparison, we briefly review this case of a convex free energy.  

By Lemma~\cite[Lemma 1]{EJM:10386443}, for any strictly convex function~$h$, the generalized Boltzmann system
\[
h_{c_1}({\bf c})+z_1\phi=h_{c_1}({\bf \bar c}),\qquad h_{c_2}({\bf c})+z_2\phi=h_{c_2}({\bf \bar c}),
\] 
uniquely defines the functions~$\{c_1(\phi,{\bf \bar{c}}),c_2(\phi,{\bf \bar{c}})\}$. 
\begin{equation}\label{eq:Poisson_autonomous}
\left[\begin{array}{c}\phi\\E\end{array}\right]_x
=
\left[\begin{array}{c}
E\\
-{\bf z}^T{\bf c}(\phi;\bar{\bf c})-\rho_0
\end{array}
\right],
\end{equation}
where the dependence of~${\bf c}=(c_1,c_2)$ upon~$\phi$ depends on~$h(c_1,c_2)$ via the corresponding generalized Boltzmann system.
The system~\eqref{eq:Poisson_autonomous}, together with~\eqref{eq:genPB_Boltzmann_nondimensional}, uniquely defines a stationary solution~$(c_1(x),c_2(x),\phi(x))$ of the corresponding generalized PNP equations~\cite{EJM:10386443}.  Furthermore, in~\cite{EJM:10386443}, we have shown that, for all strictly convex~$h$, the dynamical systems~\eqref{eq:Poisson_autonomous} are locally homeomorphic.  In this sense, all generalized Poisson-Boltzmann systems with strictly convex~$h$, ``behave the same''.  In particular, at the limit of zero applied voltage~$\phi_{\pm L}\to0$, the solutions of all such generalized Poisson-Boltzmann systems, subject to the constraint~\eqref{eq:chargeConservation}, tend to the homogenous bulk state~${\bf c}\equiv(\bar{c}_1,\bar{c}_2)$~\cite[Section 3.1]{EJM:10386443}.
Trivially, these stationary solutions are minimizers of the energy~$\mathcal{E}$, and therefore are stable under Nernst-Planck dynamics which are (Wasserstein) gradient flows along~$\mathcal{E}$.  

In the case of the PNP-steric system, see~\eqref{eq:h_of_stericPNP},
\begin{equation}\label{eq:Hessian_G_relation}
\mbox{diag}({\bf c})\mbox{Hessian}_{{ \bf c}}(h)=I+\mbox{diag}({\bf c}){\bf G}.
\end{equation}
\begin{lemma} \label{lem:Gpsd}
The PNP-steric system~\eqref{eq:stericPNP} has a corresponding strictly convex free energy,~$h$, for all~${\bf c}>0$ if and only if~${\bf G}$  is a positive semi-definite matrix.
\end{lemma}
\begin{proof}
Note that the Lemma also applies to the case of multi-species~$N\ge2$.

If~$G$ is a positive semi-definite matrix (p.s.d.), then by~\eqref{eq:Hessian_G_relation},~$\mbox{Hessian}_{{\bf c}}(h)$ is a positive definite matrix, hence~$h$ is strictly convex.

Assume that~$G$ is not p.s.d., then~${\bf G}$ has at least one negative eigenvalue~$\lambda_-<0$ with corresponding eigenvector~$v_-$, i.e.,
\[
{\bf G} v_-=\lambda_-v_-,\quad \lambda_-<0.
\]
Let~${\bf c}_-=-(1/\lambda_-,\cdots,1/\lambda_-)$.
Then,
\[
\mbox{diag}({\bf c}_-){\bf G}v_-=-v_-
\]
Hence,~${\bf I}+\mbox{diag}({\bf c}_-){\bf G}$ is singular, where~$v_-\in \ker\{\mbox{Hessian}_{{ \bf c}}(h)\}$, implying that~$h$ is not strictly convex for all~${\bf c}>0$, see~\eqref{eq:Hessian_G_relation}.
\end{proof} 
\begin{figure}[ht!]
\begin{center}
\scalebox{0.9}{\includegraphics{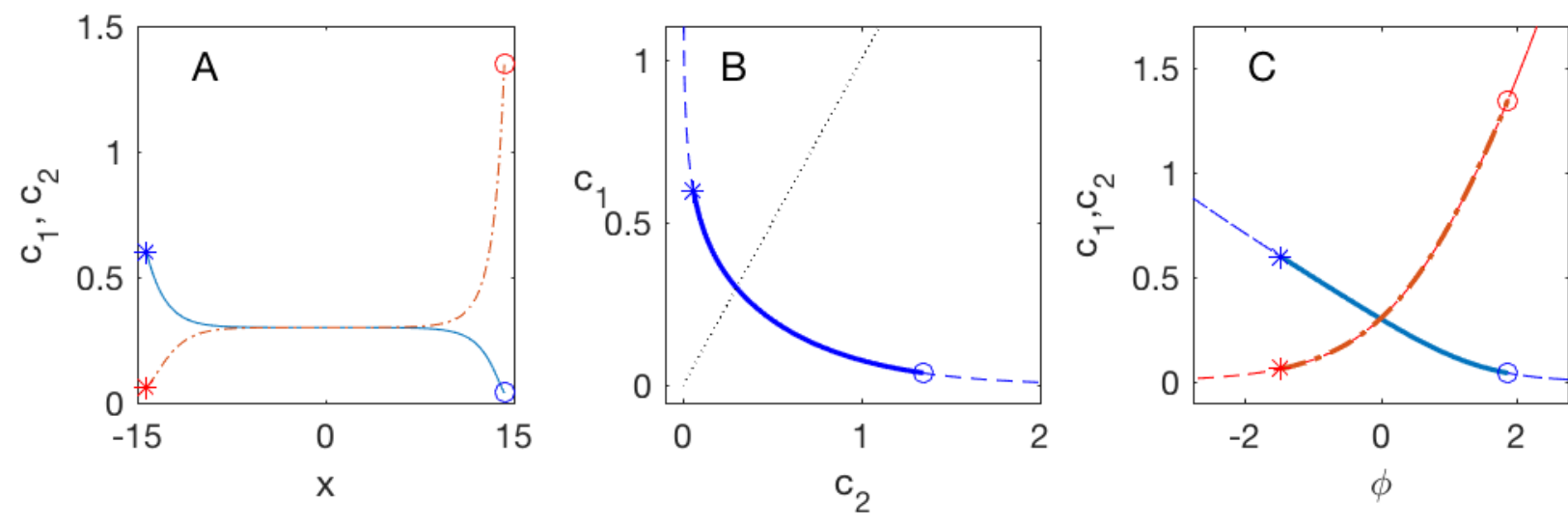}}
\end{center}
\caption{Solutions of the PB-steric system\protect\footnotemark~\eqref{eq:stericPB}
in the domain~$[-L,L]$ where~$L=14.25$,~$z_1=1=-z_2$,~$g_{11}=3.4$,~$g_{22}=0.6$, ~$g_{12}=1$,~$\rho_0=0$,~$\phi_{-L}=-1.46$,~$\phi_L=1.86$,~$\bar{c}_1=0.3$ and~$\bar{c}_2=0.32$.
A: Solution profiles~$c_1(x)$ ({\color{blue} solid}) and~$c_2(x)$ ({\color{red} dash-dot}). 
B:   Solution profile in the~$(c_1,c_2)$ plane ({\color{blue} solid}), solution trajectory ({\color{blue} dashes}), and the curve~$\{{\bf c}|E_x=0\}$ (dotted).   C: Solution profiles~$c_1(\phi)$ ({\color{blue} solid}) and~$c_2(\phi)$ ({\color{red} dash-dot}). Dashed curves are the functions~$\{c_i(\phi)\}$ defined by the steric Boltzmann equations~\eqref{eq:genPB_Boltzmann_nondimensional}.  The profiles values at points~$x^*=-L$ and~$x^\circ=L$ are marked by $\circ$ and $*$ markers, respectively, in all graphs.}
\label{fig:convexCase} 
\normalsize
\end{figure}
Figure~\ref{fig:convexCase}A presents an example of the stationary profiles of the PNP-steric~\eqref{eq:stericPNP} in a case where~${\bf G}$ is positive-definite.   It can be seen that in the interior of the domain, and far from the boundaries, the solution reaches a homogenous state.  The profiles~$c_i(\phi)$, corresponding to the solution of Figure~\ref{fig:convexCase}A, are 
presented in Figure~\ref{fig:convexCase}C.  As expected, see~\cite[Lemma 1]{EJM:10386443}, these profiles define unique functions of~$\phi$ which satisfy~\eqref{eq:genPB_Boltzmann_nondimensional}.    
The focus of this work is on multiple stationary solutions.  Therefore, in what follows, we will consider the PNP-steric model when~$h$ is not convex.

\footnotetext{In practice, we solve numerically an associated initial value problem, see Section~\ref{sec:numerics} for details, and extract the relevant parameters of the boundary value problem~\eqref{eq:stericPB} from its solution.  Accordingly, many of the parameters are not round numbers.\label{footnote:numerics}}

\section{Non-convex energetic landscape of the PNP-steric system}\label{sec:nonConvex}
Let us consider stationary solutions of the PNP-steric system~\eqref{eq:stericPNP} with counter-charged two species in the non-convex case, i.e., the case~$|{\bf G}|<0$, see Lemma~\ref{lem:Gpsd}.
In what follows, it becomes useful to consider the differential form of the steric Boltzmann equations~\eqref{eq:genPB_Boltzmann_nondimensional}.  By differentiation of~\eqref{eq:genPB_Boltzmann_nondimensional} by~$x$ and formal isolation\footnote{We consider below the cases where the system is not invertible, i.e.,~$D=0$, see, e.g., Lemma~\ref{lem:D=0}.} of~$c_{1,x}$ and~$c_{2,x}$, 
the PB-steric system~\eqref{eq:stericPB} reads as 
\begin{subequations}\label{eq:StericPB_inverse}
\begin{equation}\label{eq:E}
E_x=-[z_1(c_1-\bar{c}_1)+z_2(c_2-\bar{c}_2)],
\end{equation}
\begin{equation}\label{eq:cix}
\begin{split}
c_{1,x}&=-\frac{z_1(1+c_2g_{22})-z_2c_2g_{12}}{c_2 D}E,\\
c_{2,x}&=-\frac{z_2(1+c_1g_{11})-z_1c_1g_{12}}{c_1 D}E,
\end{split}
\end{equation}
subject to the boundary conditions~\eqref{eq:genPB_nondimensional_BC} and the mass constraints~\eqref{eq:chargeConservation}
where~$E:=\phi_x$ and~$D=D(c_1,c_2;g_{11},g_{12},g_{22})$ is the determinant of the Hessian of~$h$ 
\begin{equation}\label{eq:D}
D=|\mbox{Hessian}_{{\bf c}}(h)|=(1/c_1+g_{11})(1/c_2+g_{22})-g_{12}^2=\frac1{c_1c_2}+\frac{g_{11}}{c_1}+\frac{g_{22}}{c_2}+|{\bf G}|.
\end{equation}
\end{subequations}

In contrast to the convex case, the solutions profiles~${\bf c}(\phi)=\{c_1(\phi,{\bf \bar{c}}),c_2(\phi,{\bf \bar{c}})\}$ may not be functions of~$\phi$.  Therefore, rather than focusing on~${\bf c}(\phi)$, we  take a different approach, and consider the solutions trajectories in the~$(c_1,c_2)$ plane.  By equations~\eqref{eq:cix}, a solution trajectory passing through a point~$(c_1^0,c_2^0)$ satisfies
\begin{equation}\label{eq:Gamma_c}
\frac{dc_1}{dc_2}=\frac{z_1\left(\frac1{c_2}+g_{22}\right)-z_2g_{12}}{z_2\left(\frac1{c_1}+g_{11}\right)-z_1g_{12}},\qquad
c_1(c_2^0)=c_1^0.
\end{equation}
The following Lemma characterizes the solution trajectories.
\begin{lemma}\label{lem:dc1dc2}
Let~${\bf c}_0=(c_1^0,c_2^0)\in R_+^2$m and let~$z_1,z_2\in R$ that satisfy~\eqref{eq:valence}.
Then,
\begin{enumerate}
\item There exists a unique solution trajectory~$\Gamma_{{\bf c}_0}$, defined by~\eqref{eq:Gamma_c}, that passes through~${\bf c}_0$, and which remains in the quarter plane~$R_+^2$.
\item The solution trajectory~$\Gamma_{{\bf c}_0}$ intersects with the line~$\{{\bf c}|E_x=0\}$ at a single point.  
\item Let~$D({\bf c}_0)$ be defined by~\eqref{eq:D}.  If~$D({\bf c}_0)<0$, then the solution trajectory intersects with the curve~$\{{\bf c}|D=0\}$ at least at two points.
\end{enumerate}
\end{lemma}
\begin{proof}
Since~$z_1>0>z_2$, see~\eqref{eq:valence}, the right-hand-side of~\eqref{eq:Gamma_c} is strictly negative.  Therefore,~$c_1\le c_1^0$ for~$c_2\ge c_2^0$, implying that
\begin{equation}\label{eq:bound_dc1_dc2}
m\,c_1\le  \frac{dc_1}{dc_2}\le  M c_1<0,\qquad c_2>c_2^0,\qquad m=\frac{z_1}{z_2}\left(\frac1{c_2^0}+g_{22}\right)-g_{12},\quad M=\frac{z_1 g_{22}-z_2g_{12}}{z_2\left(1+g_{11}c_1^0\right)-z_1g_{12}c_1^0}.
\end{equation}

\begin{enumerate}
\item By~\eqref{eq:bound_dc1_dc2} and the fundamental theory of ODE, Equation~\eqref{eq:Gamma_c} admits a smooth monotonically decreasing solution for any initial condition~$c_1(c_2^0)=c_1^0$.  Further by Gronwell's Lemma, for~$c_2\ge c_2^0$, the solution is bounded by
\begin{equation}\label{eq:bound_tail}
0<c_1^0 \exp[m(c_2-c_2^0)]\le c_1(c_2)\le c_1^0 \exp[M(c_2-c_2^0)].
\end{equation}
The solution is monotonically decreasing, hence~$c_1(c_2)\ge c_1^0$ for~$c_2<c_2^0$.  Therefore,~$c_1(c_2)>0$ for all~$c_2>0$.
\item The curve~$\Gamma_E(c_1)$ along which~$E_x=0$ is given by the strictly monotonically increasing function, see~\eqref{eq:E}, 
\[
c_1=\bar{c}_1-\frac{z_2}{z_1}(c_2-\bar{c}_2).
\]
Thus,~$c_1(c_2)-\Gamma_E(c_2)$ is a strictly monotonically decreasing function.
By~\eqref{eq:bound_tail} and since~$m,M<0$,
\begin{equation}\label{eq:c2_inf}
\lim_{c_2\to \infty} c_1(c_2)=0,
\end{equation} implying that 
\[
\lim_{c_2\to \infty}c_1(c_2)-\Gamma_E(c_2)=-\infty.
\]
 Applying the same arguments to the inverse function~$c_2(c_1)$ yields that~$\lim_{c_2\to 0^+} c_1(c_2)=\infty$.
 \[
\lim_{c_2\to 0^+}c_1(c_2)-\Gamma_E(c_2)=\infty.
\]
Therefore, the solution trajectory~$\Gamma_{{\bf c}_0}$ intersects with the line~$E_x=0$ at a single point.
\item By~\eqref{eq:c2_inf},~$\lim_{c_2\to \infty} D(c_1(c_2),c_2)=\infty$.  Therefore, continuity of~$D(c_1,c_2)$ implies that the solution trajectory~$\Gamma_{{\bf c}_0}$ intersects with the curve~$\{{\bf c}|D=0\}$ at a point on~$\Gamma_{{\bf c}_0}$ for which~$c_2^0<c_2<\infty$.  Similarly, the limit~$\lim_{c_2\to 0^+} c_\Gamma(c_2)=\infty$ implies that the solution trajectory~$\Gamma_{{\bf c}_0}$ intersects with the curve~$D=0$ at a point on~$\Gamma_{{\bf c}_0}$ for which~$0<c_2<c_2^0$.
\end{enumerate}
\end{proof}
Figure~\ref{fig:solutionTrajectory} presents numerous solution trajectories in the~$(c_1,c_2)$ plane  corresponding to different solutions of the same PB-steric system~\eqref{eq:stericPB}.  As follows from Lemma~\ref{lem:dc1dc2}, the trajectories are monotonically decreasing functions of~$c_2$ that intersect once with the line~$E_x=0$.  The significance of the intersection with the line~$E_x=0$ stems from the fact that the solution is locally electro-neutral at this point, and therefore the solution is expected to reach this point in the limit of zero applied voltage or sufficiently far from the charged boundaries.  
 \begin{figure}[ht!]
\begin{center}
\scalebox{0.9}{\includegraphics{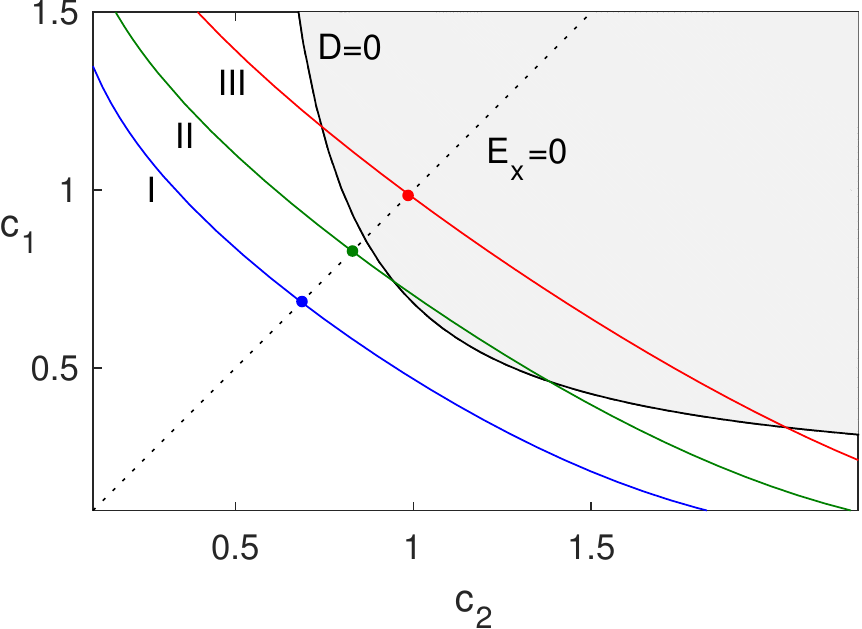}}
\end{center}
\caption{Three different solution trajectories in the~$(c_1,c_2)$ plane defined by~\eqref{eq:Gamma_c} with~$z_1=1=-z_2$,~$g_{11}=2.25$,~$g_{22}=0.75$ and~$g_{12}=2.5$. The dotted line is $\{{\bf c}|E_x=0\}$, and the dash-dotted curve is~$\{{\bf c}|D=0\}$.  Shaded region is the region~$\{{\bf c}|D<0\}$.   
}\label{fig:solutionTrajectory}
\normalsize
\end{figure}
The shaded region presented in Figure~\ref{fig:solutionTrajectory} is the region in which~$D(c_1,c_2)<0$.  As will be shown, the sign of~$D$ along the solution trajectory dictates key 
properties of the solutions.  Indeed, differentiating~\eqref{eq:E} by~$x$ and substituting~\eqref{eq:cix} in the resulting equation yields
\begin{equation}\label{eq:Exx}
E_{xx}=-z_1c_{1,x}-z_2c_{2,x}=\frac{a(x)}{D(c_1(x),c_2(x))} E,\quad a(x)=z_1^2(1/c_2+g_{22})+z_2^2(1/c_1+g_{11})-2z_1z_2g_{12}.
\end{equation}
Classical ODE theory implies that the nature of solutions of this second-order linear ODE - oscillatory or not - depends on the sign of the coefficient of~$E$.  Particularly, the sign of~$D(x)$ dictates the sign of the coefficient of~$E$ since the coefficient~$a(x)$ in equation~\eqref{eq:Exx} is strictly positive for all~$x$.

Let us distinguish between types of solution trajectories by the sign of~$D$ along the trajectory, and at the intersection point with the line~$E_x=0$: 
\begin{itemize} \item {\bf Type I} trajectories correspond to trajectories that remain in the region~$\{{\bf c}|D>0\}$, i.e., they do not intersect with the curve~$D=0$.  
\item {\bf Type III} trajectories correspond to trajectories for which~$D<0$ in a surrounding of the intersection point with the line~$E_x=0$, and thus intersect with the curve~$D=0$. 
\item Other trajectories are denoted {\bf type II} trajectories.  
\end{itemize}
Figure~\ref{fig:solutionTrajectory} presents an example of each of the trajectory type. 
\subsection{Type I solution trajectories}\label{sec:typeI}
In this section, we characterize type I solutions, and show that they qualitatively behave as solutions of generalized PNP systems in the case of convex~$h$.   
Particularly, all solutions of the PNP-steric system~\eqref{eq:stericPNP} in the convex case,~$|{\bf G}|>0$, are type I solutions, since~$D>0$ for all~$(c_1,c_2)$.  Similarly, in the dilute case,~$c_i(x)\ll1$, even if~$h$ in non-convex, the solution trajectories are type I trajectories.  This can be expected since in the dilute case the PB-steric system~\eqref{eq:stericPB} reduces at leading order to the Poisson-Boltzmann system, i.e., to the system~\eqref{eq:stericPB} with~${\bf G}\equiv 0$.  
Lemma~\ref{lem:Gpsd} implies that the concentrations of solutions corresponding to type I trajectories are bounded by~$(c_1,c_2)\in [0,-1/\lambda_-]\times [0,-1/\lambda_-]$ where~$\lambda_-$ is the negative eigenvalue of~${\bf G}$.  The following Lemma quantifies a minimal regime of solution concentrations for which solutions correspond to type I solutions.
\begin{lemma}
Let~${\bf c}_0=(c_1^0,c_2^0)\in (0,c_1^{\rm bound})\times (0,c_2^{\rm bound})$ where
\[
c_1^{\rm bound}=\frac{g_{22}}{g_{12}^2-g_{11}g_{22}},\quad c_2^{\rm bound}=\frac{g_{11}}{g_{12}^2-g_{11}g_{22}},
\]
and let~$g_{11},g_{22},g_{12}>0$ such that~$g_{11}g_{22}-g_{12}^2<0$.
Then, the solution trajectory given by~\eqref{eq:Gamma_c} satisfies for all~$x$,
\[
D(c_1(x),c_2(x))>0.
\]
\end{lemma}
\begin{proof}
It is easy to verify from that if~$c_1<c_1^{\rm bound}$ or~$c_2<c_2^{\rm bound}$ then~$D>0$, see~\eqref{eq:D}.
By Lemma~\ref{lem:dc1dc2},~$dc_1/dc_2<0$.  Thus, if~$c_2>c_2^0$ then~$c_1<c_1^0 <c_1^{\rm bound}$ and~$D>0$.  Otherwise,~$D>0$ since~$c_2 \le c_2^0<c_2^{\rm bound}$.  
\end{proof}

Figure~\ref{fig:typeI}A presents an example of the concentration profiles of  a type I solution.  The solution trajectory, corresponding to the solution of Figure~\ref{fig:typeI}A is presented in Figure~\ref{fig:typeI}B (thick solid curve). In addition, Figure~\ref{fig:typeI}B presents the trajectory~$\Gamma_{\bf \bar c}$ as defined by~\eqref{eq:Gamma_c} for any~$c_2$ (dashed curve).  As expected, the trajectory~$\Gamma_{\bf \bar c}$ remains in the region~$D>0$ for all~$c_2$ implying that~$D(c_1(x),c_2(x))>0$ for any applied potential~$V=\phi_L-\phi_{-L}$,~$\phi_L\phi_{-L}<0$.    
\begin{figure}[ht!]
\scalebox{0.9}{\includegraphics{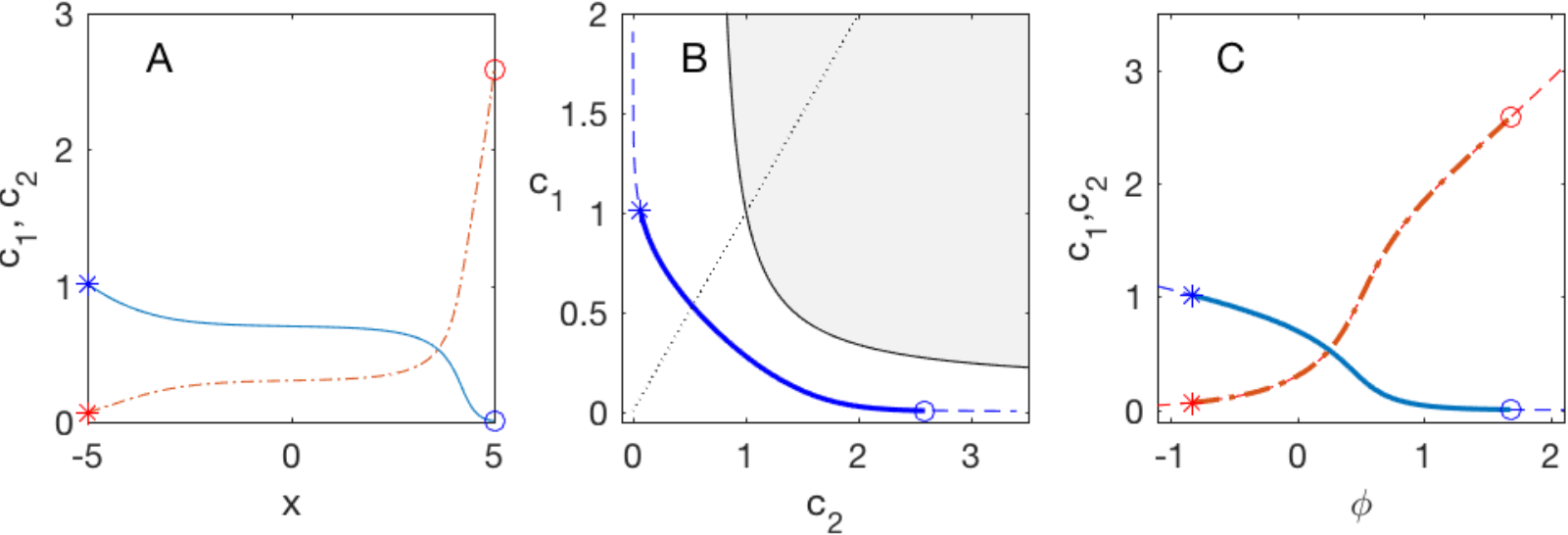}}
\caption{Solutions of the PB-steric system\protect\footnotemark[2]~\eqref{eq:stericPB} in the domain~$[-L,L]$ where~$L=5$,~$z_1=1=-z_2$,~$g_{11}=3.4$,~$g_{22}=0.6$,~$g_{12}=2.65$,~$\rho_0=0$,~$\phi_{-L}=-0.81$,~$\phi_L=1.67$,~$\bar{c}_1=0.65$, and~$\bar{c}_2=0.42$.
A: Solution profiles~$c_1(x)$ ({\color{blue} solid}) and~$c_2(x)$ ({\color{red} dash-dot}). 
B:   Solution profile in the~$(c_1,c_2)$ plane ({\color{blue} solid}), solution trajectory ({\color{blue} dashes}), and the curve~$\{{\bf c}|E_x=0\}$ (dotted).   C: Solution profiles~$c_1(\phi)$ ({\color{blue} solid}) and~$c_2(\phi)$ ({\color{red} dash-dot}). Dashed curves are the functions~$\{c_i(\phi)\}$ defined by the steric Boltzmann equations~\eqref{eq:genPB_Boltzmann_nondimensional}.  The profiles values at points~$x^*=-L$ and~$x^\circ=L$ are marked by $*$ and $\circ$ markers, respectively, in all graphs.}
\label{fig:typeI} 
\end{figure}
Intuitively, type I solutions remain in a region in which the free energy is locally convex, are therefore they are expected to behave as solutions in the convex case.  
For example, as in the convex case, the generalized Boltzmann system~\eqref{eq:genPB_Boltzmann_nondimensional} defines unique functions of~$\phi$, see, e.g., Figure~\ref{fig:typeI}C presenting the profiles~$c_i(\phi)$ which correspond to the solution of Figure~\ref{fig:typeI}A.  
In addition, as in the convex case, at the limit of zero applied voltage~$\phi_{\pm L}\to0$, type I solutions, subject to the constraint~\eqref{eq:chargeConservation}, tend to the homogenous bulk state~${\bf c}\equiv(\bar{c}_1,\bar{c}_2)$.
\begin{lemma}\label{lem:uniform_D>0}
Let~${\bf c}_0\equiv(\bar{c}_1,\bar{c}_2)$ be a homogenous state satisfying local electroneutrality
\begin{equation}\label{eq:fixedPoints}
{\bf z}^T{\bf c}_0+\rho_0=0,
\end{equation}
and 
\begin{equation}\label{eq:Dpositive_at_c0}
D({\bf c}_0)>0.
\end{equation}
Then,~${\bf c}_0$ is linearly stable under steric-Poisson-Nernst-Planck dynamics~\eqref{eq:stericPNP}.
\end{lemma}
\begin{proof}
Consider a perturbation of the homogenous state~${\bf c}_0$
\[
{\bf c}={\bf c}_0+{\bf \bm\delta}e^{ikx+\mathbf{\lambda} t},
\]
that satisfies the constraint~\eqref{eq:chargeConservation}.
Then,~$\mathbf{\lambda}$ satisfies
\begin{equation}\label{eq:disp_rel}
{\bf A}\delta=-\lambda(k){\bm\delta},\quad {\bf A}(k)=\mbox{diag}({\bf \bar c})\left[k^2\mbox{Hessian}_{{ \bf c}}(h)+{\bf z}{\bf z}^T\right].
\end{equation}
By~\eqref{eq:D} and~\eqref{eq:Dpositive_at_c0},~$\mbox{Hessian}_{{ \bf c}}(h)$ is a positive definite matrix, while~${\bf z}{\bf z}^T$ is a positive semi-definite matrix.  Therefore, for~$k\ne0$,~${\bf A}$ is positive semi-definite.  This implies that~$\lambda(k)<0$ for~$k\ne0$.  
When~$k=0$, the above system has one eigenvector~${\bm \delta}_1=\mbox{diag}({\bf \bar c}){\bf z}$ with corresponding eigenvalue
\[
\lambda_1=-{\bf z}^T\mbox{diag}({\bf \bar c}){\bf z}<0
\]
and one eigenvector~${\bm \delta}_2=(-z_2,z_1)^T$ with corresponding zero eigenvalue~$\lambda=0$.  The latter eigenvalue reflects changes in bulk concentration that maintain local electro-neutrality~\eqref{eq:fixedPoints}, but violates~\eqref{eq:chargeConservation}.  Therefore,~${\bm \delta}_2\perp {\bm \delta}$.
\end{proof}

 \subsection{Type III solution trajectories}\label{sec:typeIII}
In what follows, we will focus on the case of a concentrated solution such that 
\begin{equation}\label{eq:concentratedSolution}
D(\bar{c}_1,\bar{c}_2)<0,
\end{equation}
where~${\bar c}_i$ is the average concentration of each ionic species,
see~\eqref{eq:chargeConservation}.  

We first show that in this case the PNP-steric system~\eqref{eq:stericPNP} in free space ($L\to\infty$) admits multiple stationary solutions.  To do so, it is useful to consider the related initial value problem~\eqref{eq:StericPB_inverse} for~$x>0$ with initial conditions
\begin{equation}\label{eq:IC}
c_1(0)=c_1^0,\quad c_2(0)=c_2^0,\quad E(0)=0,
\end{equation}
rather than the boundary value problem~(\ref{eq:StericPB_inverse},\ref{eq:genPB_nondimensional_BC}).
Solutions of the initial value problem~(\ref{eq:StericPB_inverse},\ref{eq:IC}) for~$x\ge 0$, extend to the whole domain via the symmetries~$E(x)\to -E(-x)$ and~$c_i(x)\to c_i(-x)$.  
When~$ c_i^0=\bar{c}_i$,~$i=1,2$, the initial value problem~(\ref{eq:StericPB_inverse},\ref{eq:IC}) has the trivial solution~$c_i(x)\equiv \bar{c}_i$,~$i=1,2$.
We now show that this system also admits periodic solutions:  
\begin{theorem}[Existence of periodic solutions]\label{lem:periodic_solutions}
Let~${\bar{\bf c}}=(\bar{c}_1,\bar{c}_2)\in R_+^2$ such that
\begin{equation}\label{eq:D<0}
D(\bar{c}_1,\bar{c}_2;g_{11},g_{22},g_{12})<0.
\end{equation}
Then, the Poisson-Boltzmann-steric initial value problem~(\ref{eq:StericPB_inverse},\ref{eq:IC},\ref{eq:valence}) admits smooth periodic solutions.
\end{theorem}
\begin{proof}
Lemma~\ref{lem:dc1dc2} ensures that there exists a unique solution trajectory~$\Gamma_{\bar{\bf c}}=(c_1(c_2),c_2)$ 
that passes through~$\bar{\bf c}$.  
By~\eqref{eq:D<0}, and due to the continuity of~$D$ in~$R_+^2$, there exist a surrounding of~$\bar{\bf c}$ in which~$\Gamma_{\bar{\bf c}}$ does not cross~$D=0$, i.e., there exists~$\varepsilon>0$ such that
\begin{equation}\label{eq:Dnegative}
D_{\min}<D(c_1(c_2),c_2)<D_{\max}<0, \quad |c_2-\bar{c}_2|<\varepsilon.
\end{equation}
Let~$\bar{c}_2<c_2^0<\bar{c}_2+\varepsilon$.  Since~$dc_1/dc_2<0$, see Lemma~\ref{lem:dc1dc2},~$c_1^0:=c_1(c_2^0)$ satisfies
\begin{equation}\label{eq:IC0}
c_1^0<\bar{c}_1,\quad c_2^0>\bar{c}_2.
\end{equation}
\begin{figure}[ht!]
\begin{center}
\scalebox{0.6}{\includegraphics{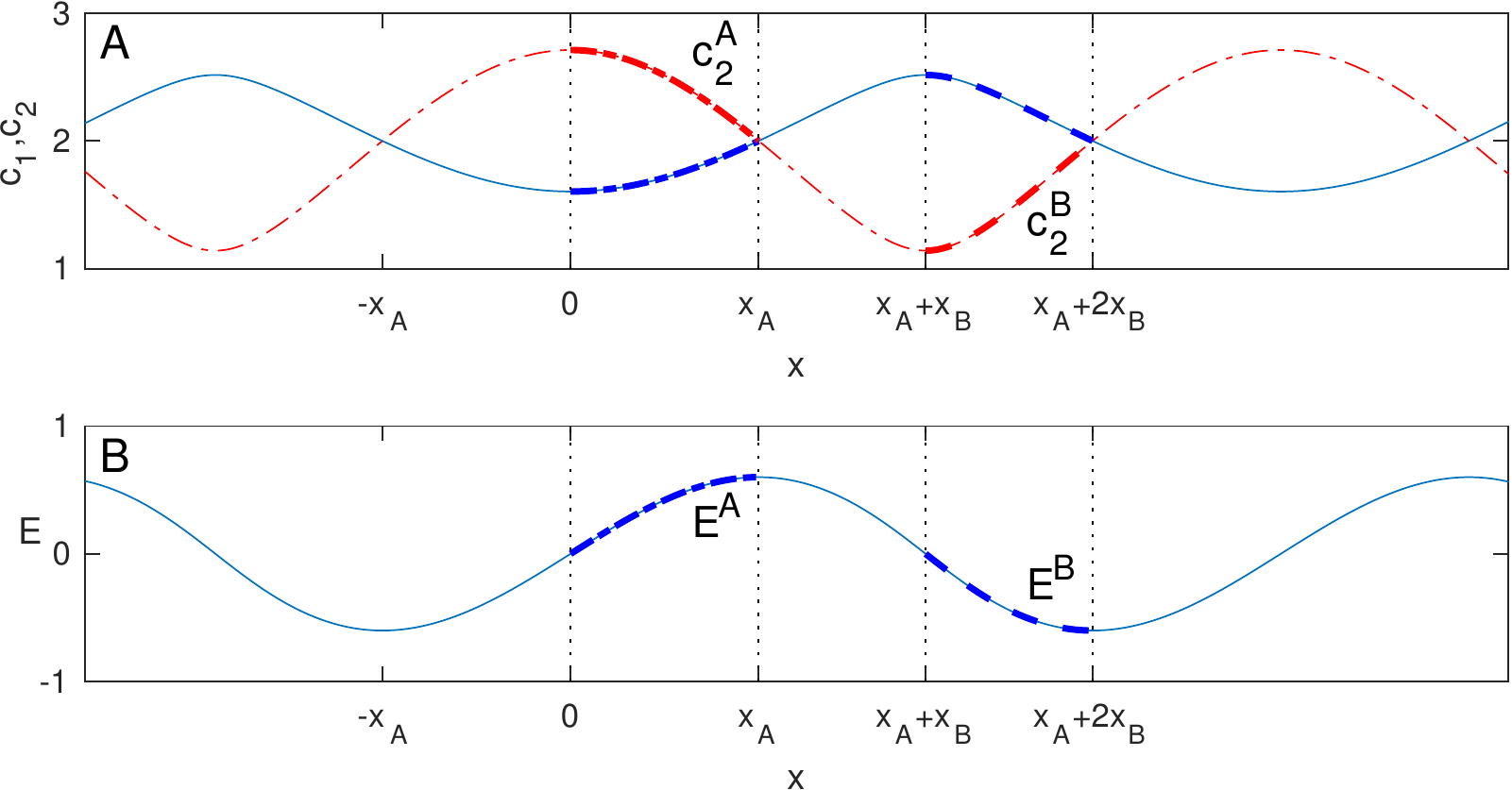}}
\end{center}
\caption{Illustration of a periodic solution of the Poisson-Boltzmann-steric system~\eqref{eq:StericPB_inverse} demonstrating the proof of Theorem~\ref{lem:periodic_solutions}.  A:~$c_1(x)$ ({\color{blue} solid}) and~$c_2(x)$  ({\color{red} dash-dots}).  B:~$E(x)$.  Thick curves in A and B correspond to the solutions~$(c_1^A,c_2^A,E^A)(x)$ (dash-dots) and~$(c_1^B,c_2^B,E^B)(x+x_A+x_B)$ (dashes) as defined in the proof of Theorem~\ref{lem:periodic_solutions}.   The two solutions extend to the whole domain via symmetries of the PB-steric~\eqref{eq:stericPB}.}
\label{fig:theorem1} 
\normalsize
\end{figure}
Denote by~$(c_1^A,c_2^A,E^A)$ the solution of~(\ref{eq:StericPB_inverse},\ref{eq:IC}).  Since~$(c_1^0,c_2^0)\in  \Gamma_{\bar{\bf c}}$, the solution trajectory corresponding to the solution~$(c_1^A,c_2^A,E^A)$ is~$\Gamma_{\bar{\bf c}}$, see Lemma~\ref{lem:dc1dc2}.
We first show that there exists a finite interval~$[0,x_A]$ where~$c_2^A(x)$ monotonically decreases from~$c_2^A(0)=c_2^0$ to~$c_2^A(x_A)=\bar{c}_2$, see Figure~\ref{fig:theorem1} for illustration.
Indeed, by~\eqref{eq:IC0},~$E_x^A(0)>0$, thus there exists a maximal surrounding~$0<x<x_A$, possibly~$x_A=\infty$, such that such tha
\begin{equation}\label{eq:positiveEx}
E_x^A(x)>0,\qquad  0<x<x_A,
\end{equation}
and~$E_x^A(x_A)=0$.   Since the solution trajectory~$\Gamma_{\bar{\bf c}}$ intersects with the line~$E_x=0$ at a single point, see Lemma~\ref{lem:dc1dc2}, then by construction~$c_i^A(x_A)=\bar{c}_i$,~$i=1,2$, and hence
 \[
c_1(x)<\bar{c}_1,\quad c_2(x)<\bar{c}_2 ,\qquad 0<x<x_A.
 \]
Furthermore, since~$D(c_1^0,c_2^0)<0$ and due to continuity of~$D$ and of the solution~$(c_1^A,c_2^A,E^A)$, there exists a maximal interval~$0<x<x_\delta\le x_A$ such that
\[
D(c_1^A(x),c_2^A(x))<0,\quad 0<x<x_\delta,
\]
and either~$D(c_1^A(x_\delta),c_2^A(x_\delta))=0$ or~$x_\delta=x_A$.
 In this surrounding,~$E^A(x)>0$, thus by~\eqref{eq:valence},
\[
c_{1,x}>0,\quad c_{2,x}<0, 
\]
see~\eqref{eq:cix}, hence
\begin{equation}\label{eq:c1c2_inequalities}
c_1^0<c_1^A(x)<\bar{c}_1,\quad c_2^0>c_2^A(x)>\bar{c}_2,\quad 0<x<x_\delta.
\end{equation}
The inequalities~\eqref{eq:Dnegative} imply that~$D<0$ for all~$ 0<x<x_\delta$.  Hence~$x_\delta=x_A$.
Finally, the inequalities~(\ref{eq:Dnegative},\ref{eq:positiveEx},\ref{eq:c1c2_inequalities}) imply that for any~$0<\varepsilon<x_A$,
\[
c_{2,x}^A(x)<M<0,\quad M=\frac{-z_2(1+\bar{c}_1g_{11})+z_1\bar{c}_1g_{12}}{c_1^0}\frac{E(\varepsilon)}{D_{\rm min}},\quad \varepsilon<x<x_A,
\]
see~\eqref{eq:StericPB_inverse}.
This implies that~$c_2^A(x)$ monotonically decreases from~$c_2^A(0)=c_2^0$ to~$c_2^A(x_A)=\bar{c}_2$ at finite time~$x=x_A<(\bar{c}_2-c_2^0)/M$.

At the point~$x=x_A$,~$E^A_x(x_A)=0$, hence~$E^A(x)$ reaches a local maximum point~$E^A(x_A)=E^A_{\max}$.
The solutions of~(\ref{eq:StericPB_inverse},\ref{eq:IC}) with the initial condition~$(c_1^0,c_2^0)\in \Gamma_{\bar \bf c}$ are parametrized by a single parameter~$c_2^0$.  When~$c_2^0=\bar{c}_2$, then~$c_2\equiv \bar{c}_2$ and~$E^A_{\max}=0$.  Continuity of the ODE solution with respect to the initial conditions implies that for any~$0\le E_{\max} \le E^A_{\max}$, there exists a solution of~(\ref{eq:StericPB_inverse},\ref{eq:IC}) with~\eqref{eq:IC0} satisfying
 \begin{equation}\label{eq:eAmax}
\tilde E(\tilde x_A)=E_{\max}.
 \end{equation}

Similarly, there exists a solution of~$(c_1^B,c_2^B,E^B)$ of~(\ref{eq:StericPB_inverse},\ref{eq:IC}) with 
\[
c_1^0>\bar{c}_1,\quad c_2^0<\bar{c}_2,\qquad (c_1^0,c_2^0)\in \Gamma_{\bar{\bf c}},
\]
such that at a point~$x=x_B<\infty$ 
\[
c_1^B(x_B)=\bar{c}_1,\quad c_2^B(x_B)=\bar{c}_2,\quad E^B(x_B)=E^B_{\max},\quad E_x^B(x_B)=0,
\]
and for any~$0\le E_{\max} \le E^A_{\max}$, there exists a solution~$(\tilde c_1^B,\tilde c_2^B,\tilde E^B)$~$(c_1^B,c_2^B,E^B)$ of~(\ref{eq:StericPB_inverse},\ref{eq:IC}) such that
 \begin{equation}\label{eq:eBmax}
\tilde E(\tilde x_B)=E_{\max}.
\end{equation}

Let~$E_{\max}=\min\{E^A_{\max},E^B_{\max}\}$, and let~$\tilde E^A$ and~$\tilde E^B$ be solutions of~(\ref{eq:StericPB_inverse},\ref{eq:IC}) satisfying~\eqref{eq:eAmax} and~\eqref{eq:eBmax}, respectively.  Then, 
\[
c_i(x)=\begin{cases}
\tilde c_i^A(x),\quad 0\le x<\tilde x_A ,\\
\tilde c_i^B(x_A+x_B-x),\quad \tilde x_A<x<\tilde x_A+\tilde x_B,\\
\end{cases}
\quad 
E=\begin{cases}
\tilde E^A(x),\quad 0\le x<\tilde x_A ,\\
 -\tilde E^B(x_A+x_B-x),\quad \tilde x_A<x<\tilde x_A+\tilde x_B,\\
\end{cases}
\]
is a smooth solution of~(\ref{eq:StericPB_inverse},\ref{eq:IC}) in the interval~$[0,L]$ where~$L=x_A+x_B$, see illustration in Figure~\ref{fig:theorem1}.
This solution is periodically extended to the whole domain by exploiting the symmetry,~$c_i(x_0-x)=c_i(x_0+x)$ where~$E(x_0)=0$, 
\[
c_i^{\rm per}(x)=\begin{cases}
c_i(x),\quad (i-1)L\le x\le iL,\qquad i \quad\mbox{is odd},\\
c_i(-x),\quad (i-1)L\le x\le iL,\qquad i \quad\mbox{is even}.\\
\end{cases}
\]
\end{proof}
Figure~\ref{fig:typeII}A presents a `periodic solution' of the PB-steric system~\eqref{eq:StericPB_inverse}.  We note that the solution is not necessarily presented in an interval in which it satisfies periodic boundary conditions.  Rather, in what follows, we refer to `periodic solutions' as solution for which there exist an interval in which they are  periodic.  As follows from Theorem~\ref{lem:periodic_solutions}, the corresponding solution trajectory lies in the domain~$\{{\bf c}|D({\bf c})<0\}$, and intersects with the line~$E_x=0$, see Figure~\ref{fig:typeII}B.  The extremum points of the solutions correspond to the end points of the solution trajectory, see~$\circ$ and~$*$ markers in Figure~\ref{fig:typeII}A and~\ref{fig:typeII}B.  The corresponding solution profiles~$c_i(\phi)$ have a regions of triple solutions, whereas the branch corresponding to the periodic solutions is bounded between two folds, see Figure~\ref{fig:typeII}C.    Particularly, the folds in~$c_i(\phi)$ correspond to points where the solution trajectory~$\Gamma_{\bar{\bf c}}$ intersects with the curve~$D=0$.  
\begin{figure}[ht!]
\begin{center}
\scalebox{0.9}{\includegraphics{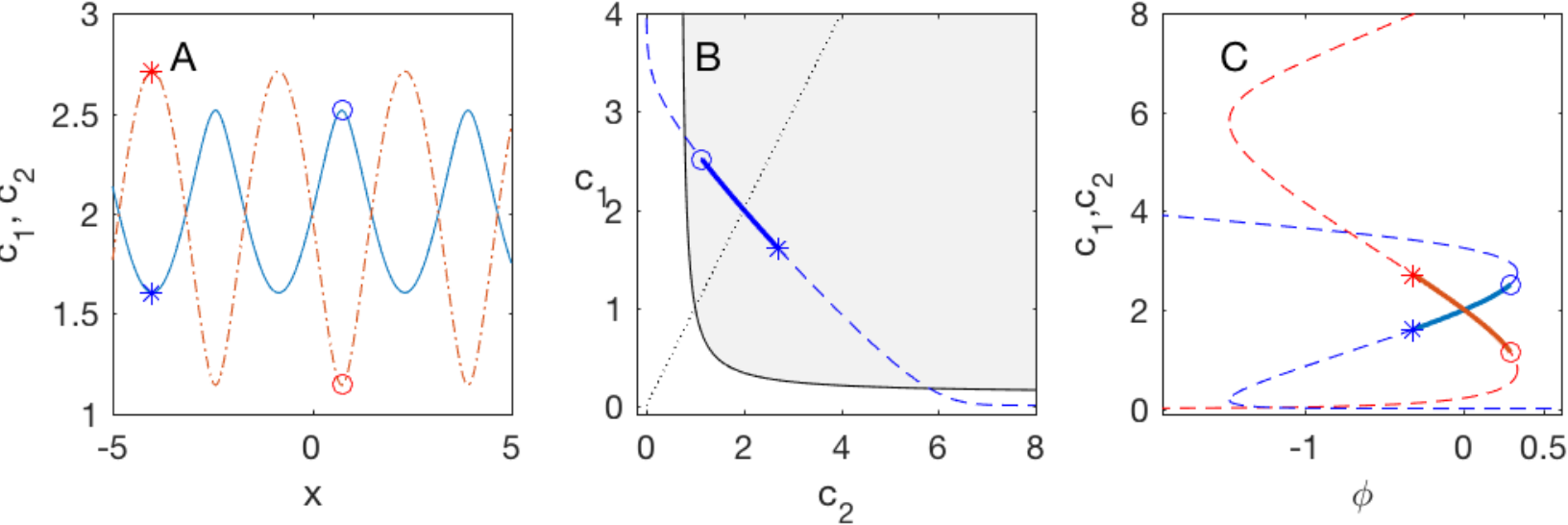}}
\end{center}
\caption{Solutions of the PB-steric system\protect\footnotemark[2]~\eqref{eq:StericPB_inverse}
in the domain~$[-L,L]$ where~$L=5$,~$z_1=1=-z_2$,~$g_{11}=3.4$,~$g_{22}=0.6$,~$g_{12}=2.65$,~$\rho_0=0$,~$\phi_{-L}=0.093$~$\phi_L=-0.19$,~$\bar{c}_1=2$, and~$\bar{c}_2=2.01$.
A: Solution profiles~$c_1(x)$ ({\color{blue} solid}) and~$c_2(x)$ ({\color{red} dash-dot}). 
B:   Solution profile in the~$(c_1,c_2)$ plane ({\color{blue} solid}), solution trajectory ({\color{blue} dashes}), and the curve~$\{{\bf c}|E_x=0\}$ (dotted).   C: Solution profiles~$c_1(\phi)$ ({\color{blue} solid}) and~$c_2(\phi)$ ({\color{red} dash-dot}). Dashed curves are the functions~$\{c_i(\phi)\}$ defined by the steric Boltzmann equations~\eqref{eq:genPB_Boltzmann_nondimensional}.  The profiles values at the maxima and minima points of~$c_2(x)$ are marked by $\circ$ and * markers, respectively, in all graphs.}
\label{fig:typeII} 
\normalsize
\end{figure}

Theorem~\ref{lem:periodic_solutions} proves the existence of periodic solutions of the initial value problem~(\ref{eq:cix},\ref{eq:Exx},\ref{eq:IC}).  These solutions maintain~$D<0$ for all~$x$.  We now show, however, that for concentrated solutions satisfying~\eqref{eq:concentratedSolution}, stationary solutions of the PNP-steric equation with large enough applied voltage~$|\phi_L-\phi_{-L}|$ can only be described by solution which cross the curve~$\{{\bf c}|D=0\}$.  Indeed, near a sufficiently charged wall,~$D>0$: Isolating~$\phi$ in the equation for~$c_1$ in~\eqref{eq:genPB_Boltzmann_nondimensional} yields
\[
z_1\phi=-\log c_1-g_{11}c_1-g_{12}c_2+\mbox{Const}.
\]
Applying the method of dominant balance for~$\phi\gg1$, yields that since~$z_1>0$ and~$c_i>0$, the only way to satisfy the above relation is if~$z_1\phi\approx-\log c_1$.  The latter implies that~$c_1\ll1$, hence~$D(c_1,c_2)>0$.  Similarly,~$D(c_1,c_2)>0$ when~$-\phi\gg1$ since~$c_2\ll1$.
\begin{lemma}\label{lem:D=0}
Under the conditions of Lemma~\ref{lem:periodic_solutions}, there exists a smooth stationary solution~$(c_1(x),c_2(x),\phi(x))$ of~\eqref{eq:stericPNP} that crosses~$D=0$, i.e.,~$D(c_1(x),c_2(x))=0$ at a point~$x\in (-L,L)$.
\end{lemma}
\begin{proof}
Let us consider a point~${\bf c}^*=(c_1^*,c_2^*)$ on the curve~$\{{\bf c}|D=0\}$ which satisfies~$c_1^*>\bar{c}_1$ and~$c_2^*<\bar{c}_2$.
Consider the PB-steric initial value problem~(\ref{eq:cix},\ref{eq:IC}) with~$(c_1^0,c_2^0)=(c_1^*,c_2^*)$.  Then,
\[
\lim_{c_2\to c_2^*}\frac{E}{D}=\lim_{c_2\to c_2^*}\frac{E_x}{D_x}=f(c_1^*,c_2^*)\lim_{c_2\to c_2^*}\frac{D}{E},
\]
where
\[
f(c_1,c_2)=-\frac{z_1(c_1-\bar{c}_1)+z_2(c_2-\bar{c}_2)}{[z_1(1+c_2g_{22})-z_2c_2g_{12}](1+g_{22}c_2)+[z_2(1+c_1g_{11})-z_1c_1g_{12}](1+g_{11}c_1)}c_1^2 c_2^2.
\]
Let us choose the intersection point~$(c_1^*,c_2^*)$ such that~$c_1^*$ is sufficiently large so that~$f(c_1^*,c_2^*)>0$.  In this case,
\[
\lim_{c_2\to c_2^*}\frac{E}{D}=\sqrt{f(c_1^*,c_2^*)}.
\]
This implies that~$c_{1,x}(c_1^*,c_2^*)$ and~$c_{2,x}(c_1^*,c_2^*)$ exists, and thus Equation~\eqref{eq:cix} has a smooth solution at a surrounding~$[-\varepsilon,\varepsilon]$ of~$x=0$ which crosses~$D=0$.
\end{proof}

Lemmas~\ref{lem:periodic_solutions} and~\ref{lem:D=0} prove the existence of a family of stationary solutions of the PNP-steric equation~\eqref{eq:stericPNP}.  Let us now discuss the stability of these solutions.  Clearly, solutions that cross~$\{{\bf c}|D=0\}$ are unstable.  Indeed, let us consider a smooth solution~${\bf u}$ of the initial value problem~(\ref{eq:cix},\ref{eq:Exx},\ref{eq:IC}) that crosses the curve~$\{{\bf c}|D=0\}$ at point~$(c_1^*,c_2^*)$, and let us consider a solution~${\bf u}_\delta$ of the initial value problem~(\ref{eq:cix},\ref{eq:Exx},\ref{eq:IC}) with the perturbed initial condition~$(c_1^0,c_2^0)=(c_1^*,c_2^*)+\vec\delta$ such that~$D(c_1^0,c_2^0)\ne0$.  Lemma~\ref{lem:dc1dc2} ensures that the trajectory associates with~${\bf u}_\delta$ will intersect with the curve~$\{{\bf c}|D({\bf c})=0\}$.  However, since~$E(x=0)=0$ and~$E_x(x=0)\ne 0$, then for small enough~$\vec\delta$, at the intersection point of the solution trajectory with the curve~$\{{\bf c}|D=0\}$, the solution will satisfy~$E\ne0$.  Thus, as the solution approaches the intersection point where~$D=0$, its gradient blows up.
\begin{figure}[ht!]
\begin{center}
\scalebox{0.9}{\includegraphics{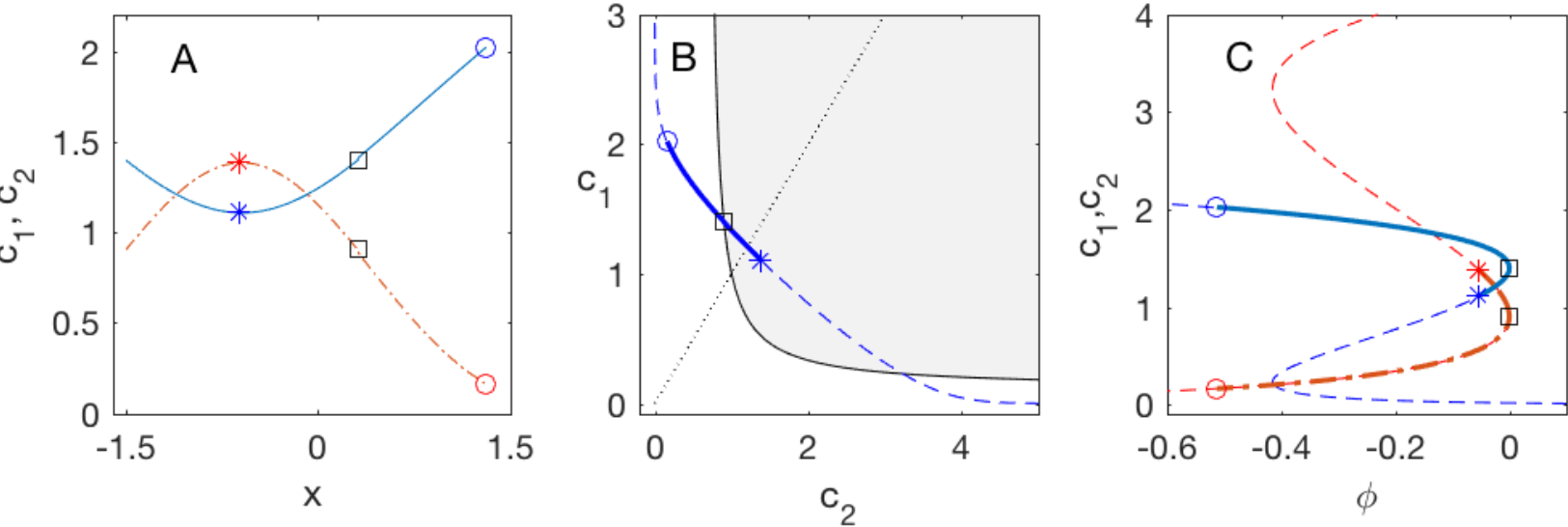}}
\end{center}
\caption{Solutions of the PB-steric system\protect\footnotemark[2]~\eqref{eq:stericPB} in the domain~$[-L,L]$ where~$L\approx 1.4$,~$z_1=1=-z_2$,~$g_{11}=3.4$,~$g_{22}=0.6$ and~$g_{12}=2.65$,~$\phi_{-L}=0$~$\phi_L=-0.51$,~$\bar{c}_1=1.39$ and~$\bar{c}_2=0.95$.
A: Solution profiles~$c_1(x)$ ({\color{blue} solid}) and~$c_2(x)$ ({\color{red} dash-dot}). 
B:   Solution profile in the~$(c_1,c_2)$ plane ({\color{blue} solid}), solution trajectory ({\color{blue} dashes}), and the curve~$\{{\bf c}|E_x=0\}$ (dotted).  C: Solution profiles~$c_1(\phi)$ ({\color{blue} solid}) and~$c_2(\phi)$ ({\color{red} dash-dot}). Dashed curves are the functions~$\{c_i(\phi)\}$ defined by the steric Boltzmann equations~\eqref{eq:genPB_Boltzmann_nondimensional}. The profiles values at points~$x^\Box\approx 0.3$,~$x^\circ=L$ and~$x^*\approx -0.6$ are marked by~$\Box$, $\circ$ and * markers, respectively, in all graphs.}
\label{fig:crossD} 
\normalsize
\end{figure}
The sensitivity of the above solution to changes in the potential~$\phi$ is also evident when considering the derivatives of the concentration profiles with respect to~$\phi$
\[
\frac{dc_1}{d\phi}=-\frac{z_1(1+c_2g_{22})-z_2c_2g_{12}}{c_2 D},\quad \frac{dc_2}{d\phi}=-\frac{z_2(1+c_1g_{11})-z_1c_1g_{12}}{c_1 D},
\]
see~\eqref{eq:cix}.  Thus, 
\[
\lim_{D\to 0}\left|\frac{dc_i}{d\phi}\right|=\infty.
\] 
\begin{remark}
Lemma~\ref{lem:D=0} and the discussion below also applies to type II solutions.  These correspond to trajectories for which~$D>0$ in a surrounding of the intersection point with the line~$E_x=0$, but do intersect with the curve~$D=0$.  Therefore, while for low applied voltage type II solutions resemble type I solutions, at higher applied voltages the solution trajectory will cross the curve~$D=0$.  At which case, the corresponding smooth solution, if exists, will be unstable.
\end{remark}

Finally, we consider the stability of homogenous stationary solutions~$c_i(x)\equiv\bar{c}_i$ of the PNP-steric system~\eqref{eq:stericPNP}.
\begin{lemma}\label{lem:uniform_instability}
Let~${\bf c}_0\equiv(\bar{c}_1,\bar{c}_2)$ be a homogenous state satisfying local electroneutrality~\eqref{eq:fixedPoints}
\begin{equation*}
{\bf z}^T{\bf c}_0+\rho_0=0,
\end{equation*}
and 
\[
D({\bf c}_0)<0.
\]
Then,~${\bf c}_0$ is linearly unstable under Poisson-Nernst-Planck dynamics~\eqref{eq:stericPNP}.
\end{lemma}
\begin{proof}
Consider a perturbation of the homogenous state~${\bf c}_0$
\[
{\bf c}={\bf c}_0+{\bf \bm\delta}e^{ikx+\mathbf{\lambda} t}.
\]
that satisfies the constraint~\eqref{eq:chargeConservation}.
Following the proof of Lemma~\ref{lem:uniform_D>0},~$\mathbf{\lambda}$ satisfies~\eqref{eq:disp_rel}.  In particular, for~$k\gg1$,
\begin{equation}\label{eq:eigvalueProblem}
{\bf A}\delta=-\lambda(k){\bm\delta},\quad {\bf A}(k)\approx k^2\mbox{diag}({\bf \bar c})\mbox{Hessian}_{{\bf c}_0}(h).
\end{equation}
Since~$D({\bf c}_0)=|\mbox{Hessian}_{{\bf c}_0}(h)|<0$ and since~${\rm tr}(\mbox{Hessian}_{{\bf c}_0}(h))>0$, the Hessian matrix has a negative eigenvalue~$\lambda_-<0$.  Thus,
\[
\lambda(k)\sim k^2|\lambda_-|,
\]
see Figure~\ref{fig:linearStability},
implying that the problem is linearly unstable.  
\begin{figure}[ht!]
\begin{center}
\scalebox{0.75}{\includegraphics{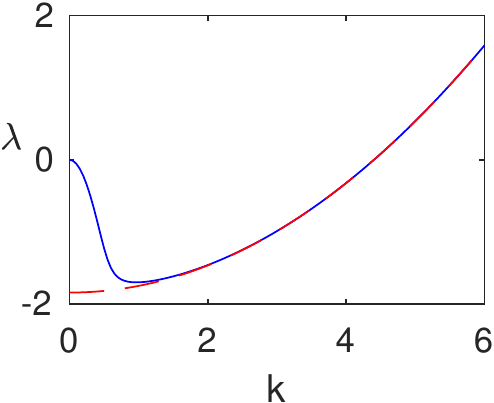}}
\end{center}
\caption{The maximal eigenvalue~$\lambda(k)$ of the eigenvalue problem~\eqref{eq:eigvalueProblem} as a function of~$k$ ({\color{blue} solid)}.  Dashed red curve is the function~$|\lambda_-|k^2$ where~$\lambda_-$ is the negative eigenvalue of the matrix~$\mbox{Hessian}_{{\bf c}_0}(h)$.}
\label{fig:linearStability} 
\normalsize
\end{figure}

\end{proof}

Lemma~\ref{lem:uniform_instability} can also be applied to the periodic solutions described in Theorem~\ref{lem:periodic_solutions}, in the case they have small enough amplitude~$|c_2^0-\bar{c}_2|\ll1$.
We, further verify by direct integration that periodic solutions on finite domain as well as periodic solutions with larger amplitudes are also unstable (data not shown). 

\subsection{Interim summary - failure of the PNP-steric at high bulk concentrations}\label{sec:interimSummary}
The study presented in Section~\ref{sec:nonConvex} considers solutions of the PNP-steric system~\eqref{eq:stericPNP} in two distinct cases.  In the dilute case, particularly when~$D(\bar{c}_1,\bar{c}_2)>0$, the PNP-steric system admits stable unique solutions, aka, type I solutions.  In the limit of zero applied voltage~$\phi_{\pm L}\to0$, these solutions tend to the homogenous bulk state~${\bf c}\equiv(\bar{c}_1,\bar{c}_2)$, in qualitative agreement with solutions of the PNP equation and the wide class of generalized PNP equations with a convex free energy.  In the concentrated solution case,~$D(\bar{c}_1,\bar{c}_2)<0$, the PNP-steric system~\eqref{eq:stericPNP} may admit multiple stationary type III solutions, e.g., the homogenous bulk state~${\bf c}\equiv(\bar{c}_1,\bar{c}_2)$ and periodic solutions.  These solutions, however, are all unstable { as higher and higher Fourier-modes become dominant.}  {\footnote{Type II solutions are a borderline case between type I and type III, they are well-defined and stable at low applied voltages, but not at higher ones.}  

To understand why the PNP-steric system~\eqref{eq:stericPNP}  becomes unstable and even ill-posed at high bulk concentrations let us first consider the classical PNP model~\eqref{eq:stericPNP_intro} with~$g_{ij}=0$.  Intuitively, in classical PNP theory, the entropy regularizes the solution and drives the system toward a spatially homogenous state.  Indeed, in the limit of zero charge density~${\bf z}\to0$, the PNP system~\eqref{eq:stericPNP_intro} with~$g_{ij}=0$, reduces to the heat equation and thus tends toward the spatially homogenous solution~${\bf c}\equiv\bar{\bf c}$.  
Steric effects, however, compete and may even dominate entropic effects.  Consider, for example, the PNP-steric model in a domain~$[-1,1]$ with~$g_{11}=g_{22}=0$, then the free energy takes the form, see~\eqref{eq:genPB_free_energy_with_L},
\[
\mathcal{E}=\mathcal{E}_{\rm entropy}+\mathcal{E}_{\rm electrostatic}+\mathcal{E}_{\rm steric},
\] 
where
\[
\mathcal{E}_{\rm entropy}=\sum_{i=1}^2\int_{-1}^{1} c_i\left(\ln \frac{c_i}{\bar{c}_i}-1\right)\, dx,\quad \mathcal{E}_{\rm electrostatic}=\int_{-1}^1 \frac{\epsilon}{2}|\nabla \phi|^2\,dx,\quad \mathcal{E}_{\rm steric} =2g_{12}\int_{-1}^1 c_1(x)c_2(x) \,dx.
\]
In the case of the spatially homogenous solution~${\bf c}_{\rm uniform}\equiv\bar{\bf c}$, the entropic and electrostatic contributions of the energy both vanish, while~$\mathcal{E}_{\rm steric}=4g_{12}\bar{c}^2$.  
Overall, the PNP-steric free energy of the homogenous solution scales as
\[
\mathcal{E}[{\bf c}_{\rm homogenous}]\sim g_{12}\bar{c}^2.
\]
In contrast, the `strongly segregated' pattern~${\bf c}_{\rm segregated}$ with segregation pattern frequency~$n$
\begin{equation}\label{eq:segregated_example}
c_1(x)=\begin{cases}
2\bar{c}& \frac{2k}{2n}<x<\frac{2k+1}{2n},\\
0&\frac{2k+1}{2n}<x<\frac{2k+2}{2n},
\end{cases}\quad c_2(x)=2\bar{c}-c_1(x),\qquad k=0,\cdots,n-1,
\end{equation}
yields~$\mathcal{E}_{\rm steric}=0$ since~$c_1(x)c_2(x)=0$.  Furthermore, direct integration shows that
\[
\mathcal{E}_{\rm electrostatic}=\frac{\bar{c}}{2n}, \quad \mathcal{E}_{\rm entropy}= \alpha\,\bar{c},\quad \alpha\approx 1.77,
\]
so that the overall free energy of the segregated pattern~\eqref{eq:segregated_example} is
\[
\mathcal{E}[{\bf c}_{\rm segregated}]=\left(\alpha+\frac1{2n}\right)\bar{c}.
\]
Clearly, the segregated pattern is energetically preferential over the homogenous solution at high concentrations,~$\bar{c}\gg \alpha/g_{12}$.  
Significantly, electrostatic contributions to the energy as well as the overall free energy decrease with increasing pattern frequencies.  The above example introduces functions with discontinuities.  It is possible to find a smooth solution which is arbitrarily close (e.g., in~$L^2$ norm) to~${\bf c}_{\rm segregated}$.  This solution would clearly have large gradients at the discontinuity points~$\frac{k}{2n}$ of~${\bf c}_{\rm segregated}$.  Nevertheless, since~$\mathcal{E}$ does not impose any energetic penalty for large gradients, one may expect that $\mathcal{E}[{\bf c}_{\rm smooth}]\approx \mathcal{E}[{\bf c}_{\rm segregated}]$. 
The emerging picture is that at high bulk concentrations, steric effects dominate entropic effects.  Consequently, entropy fails to regularize the solution and the system becomes { ill-posed} as it is pushed toward higher and higher frequencies.  
This nature of the PNP-steric system becomes apparent in the proof of Lemma~\ref{lem:uniform_instability}.
\section{The PNP-Cahn-Hilliard  (PNP-CH) model}\label{sec:PNPCH}
The failure on the PNP-steric~\eqref{eq:stericPNP} model in the case of high bulk concentrations where entropic contributions are insufficient to efficiently regularize the solution suggests that the PNP-steric model is missing terms that become dominant in the presence of large gradients.  In this section, we introduce
high-order terms that directly regularize the solution giving rise to the PNP-Cahn-Hilliard model. 
\subsection{Model derivation}\label{sec:PNPCH_derivation}
Let us consider the action potential
\[
\mathcal{A}=\mathcal{A}_{\rm PNP-steric}+\frac{\sigma}2\int_\Omega \sum_{i=1}^N|{\nabla c_i}|^2 \,d{\bf x},
\]	
where~$\mathcal{A}_{\rm PNP-steric}$ is the action potential~\eqref{eq:stericPNP_action_potential} of the PNP-steric model~\eqref{eq:stericPNP}, and 
where~$\sigma$ is a gradient energy coefficient.  The gradient energy term~$\sum_{i=1}^N|\nabla c_i|^2$ introduces a direct energetic penalty to large gradients, and thus regularizes the solution.  
This term arises naturally in Cahn-Hilliard (Ginzburg-Landau) mixture theory~\cite{cahn1958free}, and has been recently used to model steric interactions in ionic liquids~\cite{gavish2016theory}.   Moreover, while the steric interaction term~${\bf c}^T{\bf G}{\bf c}$, see~\eqref{eq:stericPNP_action_potential}, is derived by a leading order local approximation of the Lennard-Jones interaction potential~\cite{lin2014new,Horng:2012io}, the high-order gradient energy term seems to be tightly related to high-order terms in the expansion of the Lennard-Jones potential~\cite{Chun_Personal_Communication}. 
Since the steric and the gradient energy terms appears in the Cahn-Hilliard mixture theory, and since, as will be shown, the resulting system behaves more like a non-local Cahn-Hilliard (Ohta-Kawasaki) system than a PNP system, we will denote the resulting system as the {\em Poisson-Nernst-Planck-Cahn-Hilliard (PNP-CH) model}.

Similar to the PNP-steric, requiring that~$\phi$ is a critical point of the action yields Poisson's equation~\eqref{eq:genPNP_Poisson_dimensional}, whereas 
the Wasserstein gradient flow along the action gives rise to generalized Nernst-Planck equations for the dynamics of charge concentration 
\begin{equation}\label{eq:Nernst-Planck_CH}
\frac{dc_i}{dt}+\nabla \cdot J_i=0,\qquad J_i=-\breve{D}c_i\nabla\frac{\delta  \mathcal{A}}{\delta c_i}=-\breve{D}c_i\nabla \left(k_BT\log c_i+z_iq\phi+\sum_{i=1}^2 g_{ij}c_j-\sigma \Delta c_i\right).
\end{equation}
Introduction of the non-dimensional variables~\eqref{eq:scaling} and
\[
\tilde \sigma=\frac{\sigma}{\lambda^2 k_BT},
\]
yields the PNP-CH equation (presented after omitting the tildes)
\begin{equation}\label{eq:stericPNP-CH}
\begin{split}
&\frac{dc_1}{dt}=\frac{d}{dx}\left[(1+c_1 g_{11})c^\prime_1(x)+c_1g_{12}c^\prime_2(x)+z_1c_1\phi^\prime(x)-\sigma c_1c_{1,xxx}\right],\\
&\frac{dc_2}{dt}=\frac{d}{dx}\left[(1+c_2 g_{22})c^\prime_2(x)+c_2g_{12}c^\prime_1(x)+z_2c_2\phi^\prime(x)-\sigma c_2c_{2,xxx}\right],\\
&\phi_{xx}=-(z_1c_1+z_2c_2+\rho_0).
\end{split}
\end{equation}
We adopt the no-flux boundary conditions for the species concentration and the Dirichlet boundary conditions for the potential of the PNP-steric equation, see~\eqref{eq:genPNP_nondimensional_BC}.  The high-order terms, however, make it necessary to introduce additional boundary conditions.  The newly introduced gradient energy term is expected to become dominant only when the PNP-steric system fails, i.e., when~$D<0$, see~\eqref{eq:D}.  Since near a sufficiently charged wall,~$D>0$ (see Section~\ref{sec:nonConvex}), it is reasonable to set
\begin{equation}\label{eq:highOrderBC}
\left.\Delta c_i\right|_{x=\pm L}=0.
\end{equation}
For this choice, the stationary profiles, given by
\[
\frac{\delta\mathcal{A}}{\delta c_i}= k_BT\log c_i +q z_i \phi + \sum_{j=1}^N g_{ij} c_j-\sigma \Delta c_i=\lambda_i,\quad i=1,2,
\]
where~$\lambda_i$ are associated with mass conservation, are independent of~$\sigma$ at the boundaries.

The arising system~(\ref{eq:stericPNP-CH},\ref{eq:genPNP_nondimensional_BC},\ref{eq:highOrderBC}) accounts for the various interactions between the model components in a self-consistent way, while satisfying the second law of Thermodynamics, via the Onsager relation,
\[
\frac{d \mathcal{A}}{dt}=-\frac{\breve{D}}{k_BT}\int_\Omega \sum_{i=1}^2 c_i\left|\nabla\frac{\delta \mathcal{A}}{\delta c_i}\right|^2\, dx\le 0.
\]

\subsection{Linear stability analysis of the homogenous state}\label{sec:linearAnalysis}
Recall that in the absence of applied voltage,~$\phi_{\pm L}=0$, the homogenous stationary solution~${\bf c}\equiv \bar{\bf c}$ is stable under dynamics of all generalized PNP models with a convex free energy, see Section~\ref{sec:convex}.  In contrast, at large enough concentrations, the homogenous stationary solution is unstable under PNP-steric dynamics, see Lemma~\ref{lem:uniform_instability}.  Here, we study the stability of the homogenous stationary solution under dynamics of the PNP-CH system~(\ref{eq:stericPNP-CH},\ref{eq:genPNP_nondimensional_BC},\ref{eq:highOrderBC}) in free-space, i.e.,~$L\to\infty$ and~$\phi_{\pm L}=0$.

Let us consider a perturbation of the homogenous state~$c_i\equiv \bar{c}_i,\, i=1,2$,
\begin{equation}\label{eq:perturbed_barc}
c_1=\bar{c}_1+\varepsilon u_1 e^{ikx+\lambda t},\quad c_2=\bar{c}_2+\varepsilon v_1 e^{ikx+\lambda t},
\qquad 0<\varepsilon\ll1.
\end{equation}
Substitution of~\eqref{eq:perturbed_barc} into the PNP-CH system~\eqref{eq:stericPNP-CH} yields the linearized system
${\bf M} v=\lambda v$, where
\begin{equation}\label{eq:M}
{\bf M}(k;\bar{c}_1,\bar{c}_2,\sigma)=-\left[\begin{array}{cc}
\bar{c}_1&\\
&\bar{c}_2
\end{array}\right]
\left[\begin{array}{cc}
\left(\frac{1}{\bar{c}_1}+ g_{11}\right)k^2+z_1^2+k^4\sigma&k^2g_{12}+z_1z_2\\
k^2g_{12}+z_1z_2&\left(\frac1{\bar{c}_2}+ g_{22}\right)k^2+z_2^2+k^4\sigma
\end{array}
\right],
\end{equation}
and~$v=[u_1,v_1]^T$.
The trace of the above matrix~${\bf M}(k)$ is negative for all~$k\in R$.  Therefore, at least one of its eigenvalues in strictly negative.  Let us denote by~$\lambda(k)$ the larger eigenvalue of~${\bf M}(k)$.  

When~$k=0$, the matrix~${\bf M}$ has one zero eigenvalue~$\lambda(k=0)=0$ with corresponding eigenvector~$\vec{v}_0:=(-z_2,z_1)^T$.  This eigenvector corresponds to changes in bulk concentration that do not effect global charge neutrality, see \eqref{eq:zeroVoltageRefPoint}. 
For~$k\ne0$,~$\lambda(k)$ satisfies
\begin{equation}\label{eq:dispersion_rel}
\left(\lambda(k)+\left(\frac1{\bar{c}_1}+ g_{11}\right)k^2+z_1^2+k^4\sigma\right)\left(\lambda(k)+\left(\frac1{\bar{c}_2}+ g_{22}\right)k^2+z_2^2+k^4\sigma\right)-(k^2g_{12}+z_1z_2)^2=0
\end{equation}
Varying~$\sigma$, the instability onset is obtained by seeking~$\sigma=\sigma_c$ and~$k_c$ such that $\lambda(k=k_c,\sigma=\sigma_c)=0$, $\lambda(k\ne k_c)<0$ and~$\frac{d\lambda}{dk}(k_c)=0$.  At these conditions,~$\sigma>\sigma_c$, implies that~$\lambda(k)<0$ for all~$k\ne0$, hence the homogenous solution is stable.
Substituting~$\lambda=0$ in~\eqref{eq:dispersion_rel} yields that~$\sigma_c$ satisfies
\begin{equation}\label{eq:rel_k_sigma}
\begin{split}
&\sigma_c^2k^8+\left(\frac1{\bar{c}_1}+ \frac1{\bar{c}_2}+g_{11}+g_{22}\right)\sigma_c k^6+\sigma_c(z_1^2+z_2^2)k^4+\left[\left(\frac1{\bar{c}_1}+ g_{11}\right)\left(\frac1{\bar{c}_1}+
g_{22}\right)-g_{12}^2\right]k^4\\
&+z_2^2\left(\frac1{\bar{c}_1}+g_{11}\right)+z_1^2\left(\frac1{\bar{c}_2}+g_{22}\right)-2z_1z_2k^2g_{12}=0.
\end{split}
\end{equation}
Since all terms in the above equality are strictly positive, except possibly the fourth term, a necessary condition for instability about~$(\bar{c}_1,\bar{c}_2)$ is that
\[
\left(\frac1{\bar{c}_1}+ g_{11}\right)\left(\frac1{\bar{c}_2}+
g_{22}\right)-g_{12}^2<0,
\]
or
\[
g_{12}>g_{12}^{\rm crit},\qquad g_{12}^{\rm crit}:=\sqrt{\left(\frac1{\bar{c}_1}+ g_{11}\right)\left(\frac1{\bar{c}_2}+
g_{22}\right)}.
\]
This condition is nothing but the condition~$D(\bar{c}_1,\bar{c}_2)<0$, see~\eqref{eq:D}. 
When~\eqref{eq:rel_k_sigma} has a solution~$k_c,\sigma_c>0$, the solution of the linearized system ${\bf M} v=\lambda v$ at the bifurcation point~$\sigma=\sigma_c$ takes the form 
\begin{equation}\label{eq:lin_solution}
\left[\begin{array}{c}c_1\\c_2\end{array}\right]=\left[\begin{array}{c}\bar c_1\\\bar c_2\end{array}\right]+\varepsilon \left(\alpha\vec v_0+\beta\vec v_{k_c}e^{ik_cx}+ c.c.\right),
\end{equation}
where~$\varepsilon$ is the magnitude of the~$k=0$ and~$k=k_c$ modes of the initial condition,~$\alpha,\,\beta$ are constants, and~$\vec v_0$,~$\vec v_{k_c}$ are the normalized eigenvectors corresponding the zero eigenvalue of the matrices~${\bf M}(0)$ and~${\bf M}(k_c)$, respectively,
\[
{\bf M}(0) \vec v_0=\vec 0,\qquad {\bf M}(k_c) \vec v_{k_c}=\vec 0,\quad \|\vec v_0\|=\|\vec v_{k_c}\|=1. 
\]
The expressions for~$(k_c,\sigma_c)$ are analytically available but, in the general case, they are cumbersome.  
In the symmetric case,~$g:=g_{11}=g_{22}$,~$\bar{c}:=\bar{c}_1=\bar{c}_2$ and~$z_1=-z_2=1$, the expressions for the instability onset reduce to
\[
\sigma_c=\frac{\left(g_{12}-g_{12}^{\rm crit}\right)^2}8,\quad k_c=\frac{2}{\sqrt{g_{12}-g_{12}^{\rm crit}}},\quad g_{12}^{\rm crit}=g+\frac1{\bar{c}}.
\]

\subsection{Weakly nonlinear analysis}\label{sec:weaklyNonlinear}
To understand the nature of instability at~$\sigma=\sigma_c$, we conduct a weakly nonlinear analysis.  For simplicity, we restrict to the case~$\bar{c}:=\bar{c}_1=\bar{c}_2$.
Let us consider the slow dynamics of a solution slightly below the instability onset~$\sigma=\sigma_c-\varepsilon$.
\[
\left[\begin{array}{c}c_1(T,x)\\c_2(T,x)\end{array}\right]=\bar{\bf c}+\sqrt{\varepsilon} \vec u_0 +\varepsilon \vec u_1+\varepsilon^{3/2} \vec u_2+\cdots,\qquad T=\varepsilon t.
\]
Motivated by~\eqref{eq:lin_solution}, we consider the ansatz
\[
\vec u_0(T,x)=\alpha_0(T)\vec v_0+\beta_0(T)\vec v_{k_c}e^{ik_cx}+c.c.
\]
Note that we limit this study to spatially uniform amplitudes~$\alpha_0$ and~$\beta_0$.
In this case, conservation of mass implies that
\[
\frac{d}{dT}\int c_1 \,dx=0,\quad \frac{d}{dT}\int c_2 \,dx=0.
\]
Hence,~$\alpha_0(T)\equiv0$.\footnote{In the general case where~$\alpha_0$ is spatially varying, the nonlinear interaction between the two modes,~$k=0$ and~$k=k_c$, specifically via the spatial derivatives of~$\alpha_0$, must be taken into account, see, e.g.,~\cite{cox2001new,golovin2003self,schneider2013justification}.}

At~$O(\varepsilon)$, the solution satisfies
\[
{\bf M}(k)\vec u_1-\vec r_1(T)e^{2ik_cx}=0,\quad \vec r_1=\beta_0^2(T)\,\mbox{diag}(\vec v_{k_c})\left[\begin{array}{cc}
k_c^4\sigma+g_{11}k_c^2+z_1^2&k_c^2g_{12}+z_1z_2\\
k_c^2g_{12}+z_1z_2& k^4_c\sigma+g_{22}k_c^2+z_2^2
\end{array}
\right]\vec v_{k_c},
\]
whose solution is of the form
\[
\vec u_1=\alpha_1(T)\vec{v}_0+\beta_1(T)\vec v_{k_c}e^{ik_cx}+\vec \gamma(T)e^{2ik_cx},
\]
where~$\vec \gamma$ satisfies
\[
{\bf M}(2k_c)\vec\gamma(T)=\vec r_1(T).
\]

At~$O(\varepsilon^{3/2})$, the equations take the form
\begin{subequations}\label{eq:epsilon32}
\begin{equation}
{\bf M}(k_c)\left[\begin{array}{c}A_{21}\\B_{21}\end{array}\right]e^{ik_c x} =\left[-\frac{d\beta_0(T)}{dT}\vec{1}+\beta_0(T)\vec{a} +\beta_0^3(T)\vec{b}\right]e^{ik_c x}+\vec F[v_1,v_2]e^{2ik_c x}+\cdots.
\end{equation}
where
\begin{equation}
\begin{split}
\vec{a}&=\frac{k_c^4\bar{c}}2\left[1,-\frac{k_c^4\sigma_c+\left(g_{11}+\frac1{\bar{c}}\right)k_c^2+z_1^2}{g_{12}k_c^2+z_1z_2}\right]^T,
\\\vec{b}&=\frac18\left[
\begin{array}{cc}
-1&\\&\frac1{(g_{12}k_c^2+z_1z_2)\bar{c}}
\end{array}\right]
\left[
\begin{array}{cc}
16k_c^4\sigma_c+4g_{11}k_c^2+z_1^2+\frac{2}{\bar{c}}k_c^2&4g_{12}k_c^2+2z_1z_2\\
(4g_{12}k_c^2+z_1z_2)(k_c^4\sigma_c\bar{c}+g_{11}k_c^2\bar{c}+z_1^2\bar{c}+k_c^2)& C_{22}\end{array}
\right]\vec\gamma,
\end{split}
\end{equation}
and
\begin{equation}
\begin{split}
C_{22}=&14k_c^8\sigma_c^2\bar{c}+(14(1+(g_{11}+g_{22}/7)\bar{c}))\sigma_c k_c^6+(((14z_1^2-z_2^2)\sigma_c+2g_{22}g_{11}+2g_{12}^2)\bar{c}+2g_{22})k_c^4\\&+((-g_{11}z_2^2+4g_{12}z_1z_2+2g_{22}z_1^2)\bar{c}-z_2^2)k_c^2+cz_1^2z_2^2.
\end{split}
\end{equation}
\end{subequations}
Since~${\bf M}(k_c)$ is a singular matrix with~$\vec v_0\in \ker{\bf M}(k_c)$, Equation~\eqref{eq:epsilon32} has a solution only when
\[
v_0\perp \left[-\frac{d\beta_0(T)}{dT}\vec{1}+\beta_0(T)\vec{a} +\beta_0^3(T)\vec{b}\right].
\]
In particular, the stationary solution~$\beta_0(T)\equiv \beta_0$ satisfies
\[
\beta_0^2=-\frac{<\vec v_0,\vec a>}{<\vec v_0,\vec b>}.
\]
If~$-\frac{<\vec v_0,\vec a>}{<\vec v_0,\vec b>}>0$, the bifurcation is super-critical leading to the solution
\[
c_1(x)=\bar{c}+\sqrt{\sigma_c-\sigma}\sqrt{-\frac{<\vec v_0,\vec a>}{<\vec v_0,\vec b>}}\cos(k_c x),
\]
otherwise the bifurcation is sub-critical.
In the symmetric case,~$g_{11}=g_{22}$ and~$z_1=-z_2=1$, the analytic expression is available 
\[
-\frac{<\vec v_0,\vec a>}{<\vec v_0,\vec b>}=\frac{32\bar{c}^3}{g_{12}-g_{12}^{\rm crit}}\frac{3g_{12}-g_{12}^{\rm crit}}{3(g_{12}-g_{12}^{\rm crit})+3g_{12}+g}>0,
\]
implying that bifurcations in the symmetric case are always super-critical.

Figure~\ref{fig:bifDiagram}A presents the regions in parameter space in which the system undergoes supercritical and subcritical phase transitions.  As expected, in the symmetric case~$g_{11}-g_{22}=0$, phase transitions are always super-critical, while subcritical bifurcations emerge in the asymmetric case~$g_{11}\ne g_{22}$, near the curve~$g_{12}=g_{12}^{\rm crit}$.  Below the curve~$g_{12}<g_{12}^{\rm crit}$, the only stable solution is the homogenous state~${\bf c}\equiv\bar{\bf c}$, see Figure~\ref{fig:bifDiagram}(1).  The bifurcation diagram of the system at point `a' residing in the supercritical regime is presented in Figure~\ref{fig:bifDiagram}B.  At~$\sigma<\sigma_c$, the homogenous state becomes unstable, and rather a branch of stable periodic solutions whose amplitude is proportional to~$\sqrt{\sigma_c-\sigma}$, see Figure~\ref{fig:bifDiagram}(2), emerges.  Finally, Figure~\ref{fig:bifDiagram}C presents the bifurcation diagram of the system at point `b' which resides in the subcritical regime.  In this case, there exists a region~$\sigma_c<\sigma<\sigma_{\rm fold}$, where~$\sigma_{\rm fold}$ is the fold point, of multiple stationary which are all stable.  For example, as verified by direct simulation, at point~$\sigma=\sigma^*$, both the homogenous state (Figure~\ref{fig:bifDiagram}(1)) and the periodic state presented in Figure~\ref{fig:bifDiagram}(3) are stable under PNP-CH dynamics with periodic boundary conditions for~$\phi$.
\begin{figure}[ht!]\begin{center}
\scalebox{0.9}{\includegraphics{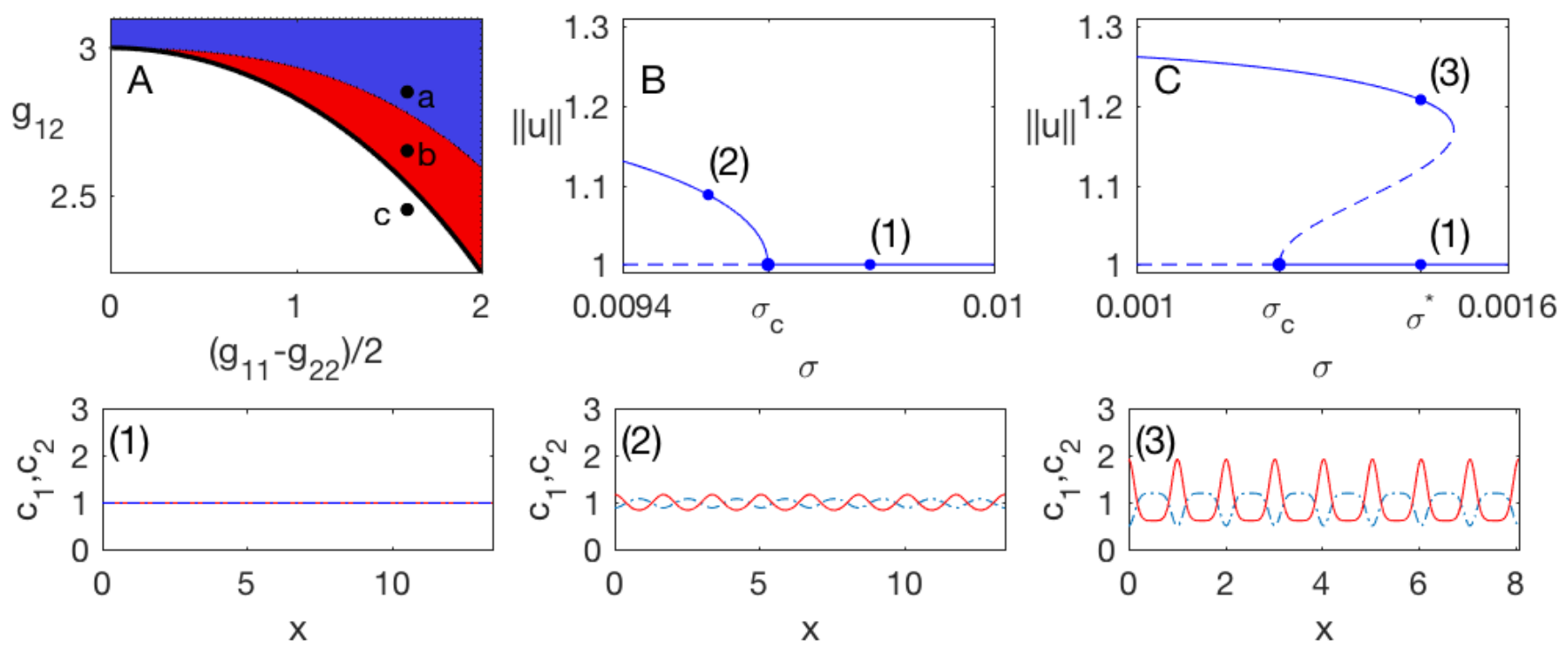}}
\end{center}
\caption{A: Parameter space depicting the regions of parameters~$(g_{11}-g_{22})/2$ and~$g_{12}$ where~$g_{11}+g_{22}=4$ and~$\bar{c}=1$ in which the system undergoes supercritical (blue region including point a) and subcritical phase transitions (red region including point b) at~$\sigma=\sigma_c$.  Solid curve is~$g_{12}=g_{12}^{\rm crit}$ below which the homogenous state~$c_i\equiv \bar{c}$, presented in graph (1), is stable.  B,C: The bifurcation diagrams corresponding to points a and b, respectively, in the parameter space.  Solid curves are stable branches, while dashed curve correspond to unstable solutions.  Graphs (1)-(3): Example plots of~$c_1(x)$ (dash-dots) and~$c_2(x)$ (solid) at points a-c of diagrams A, B are presented in the bottom graphs.}
\label{fig:bifDiagram} 
\normalsize
\end{figure}

\subsection{Effect of finite domain size}\label{sec:finiteDomain}
\begin{figure}[ht!]\begin{center}
\scalebox{0.9}{\includegraphics{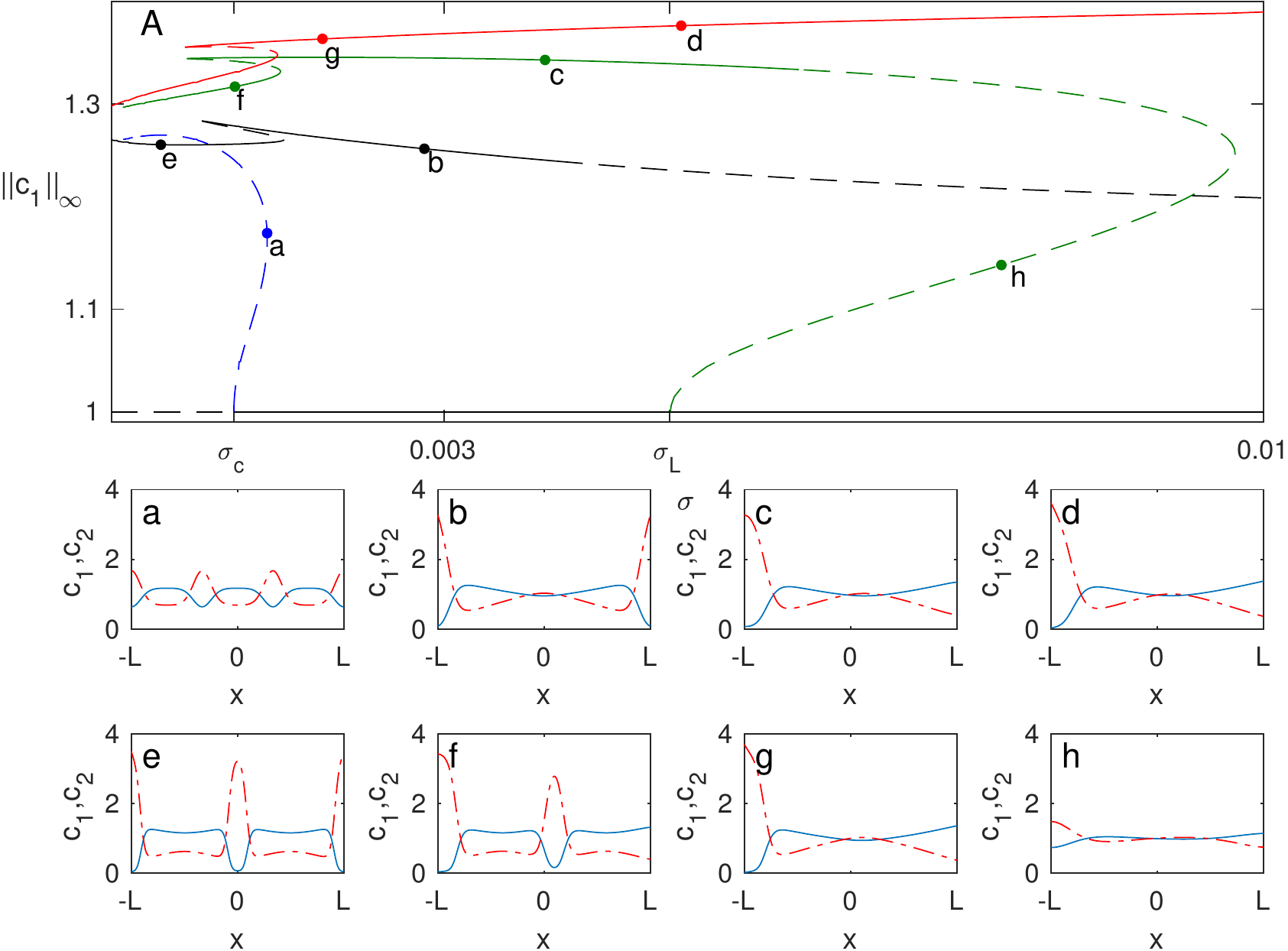}}
\end{center}
\caption{A: Bifurcation diagram of stationary solution of the PNP-CH system~(\ref{eq:stericPNP-CH},\ref{eq:genPNP_nondimensional_BC},\ref{eq:highOrderBC})  in the domain~$[-L,L]$ where~$L=3\pi/k_c$ and~$k_c\approx 6.23$ with parameters~$g_{11}=3.6,g_{22}=0.4,g_{12}=2.65$,~$z_1=1=-_2$~$\bar{c}_1=\bar{c}_2=1$. Solid curves are stable branches with respect to PNP-CH dynamics, while dashed curve correspond to unstable solutions. Plots of~$c_1(x)$ (solid blue) and~$c_2(x)$ (dash-dotted red) at points a-h of diagram A are presented in the bottom graphs a-h.}
\label{fig:bifurcationDiagramV0} 
\normalsize
\end{figure}
To study the effect of finite domain size, we conduct a numerical continuation study of the PNP-CH system~(\ref{eq:stericPNP-CH},\ref{eq:genPNP_nondimensional_BC},\ref{eq:highOrderBC}) in a finite domain~$[-L,L]$ using pde2path 2.0~\cite{dohnal2014pde2path,uecker2014pde2path}.   Unlike a typical continuation study which considers stability of the solution with respect to the (elliptic) stationary equation, we test here stability with respect to PNP-CH dynamics by linking the continuation method to a Comsol 5.2a finite element solver.  When a solution is found to be unstable, the solver returns the computed (stable) stationary solution it had converged to.  Therefore, in addition to stability information, this method allows to efficiently map distinct solution branches by studying the stationary solutions computed by the solver.  See Section~\ref{sec:numerics} for numerical details.

The bifurcation diagram of the PNP-CH~(\ref{eq:stericPNP-CH},\ref{eq:genPNP_nondimensional_BC},\ref{eq:highOrderBC})  in the domain~$[-L,L]$ where~$L=3\pi/2k_c$, and otherwise with the same parameters as in Figure~\ref{fig:bifDiagram}, is presented in Figure~\ref{fig:bifurcationDiagramV0}A.  As expected, we recover the branch of periodic solutions with period~$\pi/k_c$, see blue branch in Figure~\ref{fig:bifurcationDiagramV0}A, and a plot of a solution in point `a' along this branch in Figure~\ref{fig:bifurcationDiagramV0}a.  As noted, in free-space, solutions along this branch become stable under PNP-CH dynamics after the fold point~$\sigma=\sigma^*$, see~Figure~\ref{fig:bifDiagram}C.  In contrast, we observe that, in a finite domain, all solutions along this branch are unstable under PNP dynamics.  Rather, additional branches with corresponding stable solutions emerge.  The black branch in Figure~\ref{fig:bifurcationDiagramV0}A corresponds to periodic solutions with period~$3\pi/k_c$ (point `b', see Figure~\ref{fig:bifurcationDiagramV0}b) or with period~$3\pi/2k_c$ (point `e', see Figure~\ref{fig:bifurcationDiagramV0}e).  Two additional branches (green and red with points `c' and `d', respectively) correspond to asymmetric solutions, see, e.g., Figure~\ref{fig:bifurcationDiagramV0}c.  Note that since~$\phi=V=0$ on both boundaries, there is no preferred direction for the solutions.  Hence, every point along the green and red branches corresponds to two solutions which are reflections of each other.  Hence, for example, at~$\sigma=0.003$, we find that the system has six different stable stationary solutions: the homogenous state, the symmetric solution corresponding to point `b', two solutions corresponding to a point on the red branch and two additional ones corresponding to a point on the green branch.
Particularly, we observe multiple stable stationary solutions over the a wide range of~$\sigma$. 
\subsection{Effect of applied voltage}\label{sec:appliedVoltage}
Following the study of finite-domain effect in Section~\ref{sec:finiteDomain}, we now apply numerical continuation methods~\cite{dohnal2014pde2path,uecker2014pde2path} to study the effect of applied voltage.  

We consider the PNP-CH system~(\ref{eq:stericPNP-CH},\ref{eq:genPNP_nondimensional_BC},\ref{eq:highOrderBC}) in the domain~$[-5,5]$ with applied voltage~$\phi_L=-\phi_{-L}=1$, and otherwise with the same parameters as in Figure~\ref{fig:bifDiagram}.  The corresponding bifurcation diagram is presented in Figure~\ref{fig:bifurcation_V1} where solid curves correspond to solutions which are stable with respect to PNP-CH dynamics, and dashed curve correspond to unstable solutions.  
Figure~\ref{fig:bifurcation_V1_solutions} presents concentration profiles of solutions corresponding to points `a'-`o' along various branches presented in Figure~\ref{fig:bifurcation_V1}. 
We observe that the bifurcation diagram is significantly more complex than the one presented in Figure~\ref{fig:bifurcationDiagramV0} for solutions of a PNP-CH~(\ref{eq:stericPNP-CH},\ref{eq:genPNP_nondimensional_BC},\ref{eq:highOrderBC}) in the domain~$[-L,L]$ where~$L\approx 1.49$ and with no applied voltage~$V=0$.  As we shall see, this is mainly due to the larger domain size in Figure~\ref{fig:bifurcation_V1}.  Indeed, as expected, the solutions profiles are essentially identical near the charged boundaries, see Figure~\ref{fig:bifurcation_V1_solutions}.  The reason is that near a sufficiently charged wall, the concentration profiles near the boundaries are independent of the high-order gradient-energy terms.  Namely, boundary layers of~$O(1/\sqrt{\sigma})$ width do not develop in the concentration profiles, and rather dominant balance shows that
\[
c_1\approx e^{-z_1\phi},\quad c_2\approx -\frac{z_2\phi}{g_{22}},\qquad \phi\gg1,\qquad  {\rm or}\qquad c_2\approx e^{-z_2\phi},\quad c_1\approx -\frac{z_1\phi}{g_{11}},\qquad \phi\ll-1,
\]
see discussion before Lemma~\ref{lem:D=0}.  The different branches, therefore, describe distinct behavior in the bulk, namely, different numbers, amplitudes and locations of peaks.    Particularly, while the lowest branch in the bifurcation diagram, see Figure~\ref{fig:bifurcation_V1}, corresponds to solutions with a uniform bulk, see Figure~\ref{fig:bifurcation_V1_solutions}a, solutions `f' and `o' have one peak in the bulk, solutions `c', `l' and `m' have two peaks, and solution `h' with three peaks.  Finally, solution `j' has four peaks, see Figure~\ref{fig:bifurcation_V1_solutions}.  The maximal number of peaks is dictated by the domain size.  Thus, the bifurcation diagram on a larger domain will have more branches.  To distinguish between the number of peaks, and to distinguish between two solutions that have the same number of peaks but at different locations, e.g., solutions `l' and `m', we choose a weighted arc-length solution norm 
\begin{equation}\label{eq:weighted_arc-length}
\|c_1\|:=\int_{-L}^L \left(1+\frac{x+L}{2L}\right)\sqrt{1+|\nabla c_1|^2}\,dx,
\end{equation}
to plot bifurcations in the~$(\sigma,\|c_1\|)$ plane.

 
\begin{figure}[ht!]\begin{center}
\scalebox{0.9}{\includegraphics{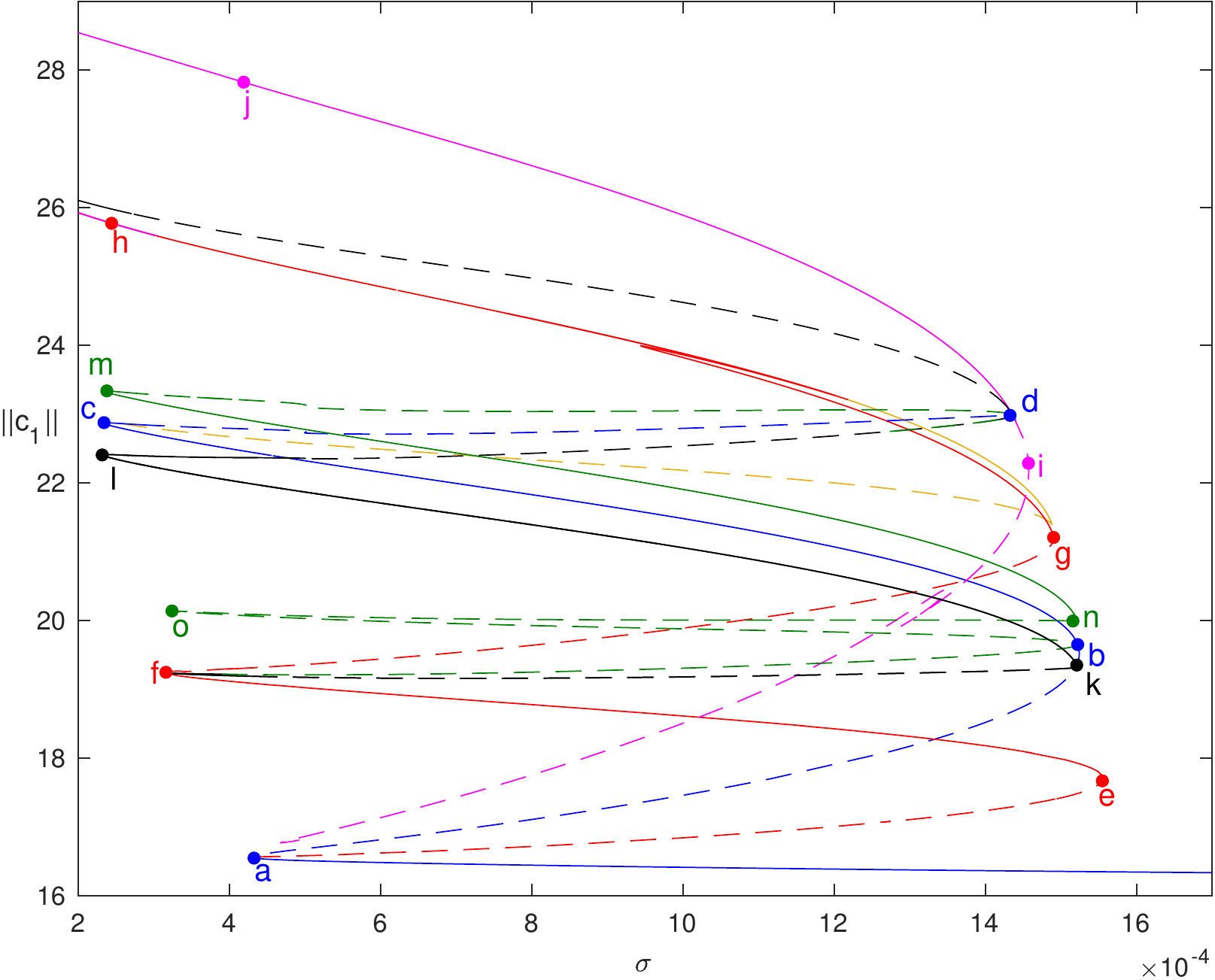}}
\end{center}
\caption{Bifurcation diagram of stationary solutions of PNP-CH system~(\ref{eq:stericPNP-CH},\ref{eq:genPNP_nondimensional_BC},\ref{eq:highOrderBC})  in the domain~$[-5,5]$ with applied voltage~$V=1$, and parameters~$g_{11}=3.6$,~$g_{22}=0.4$,~$g_{12}=2.65$,~$z_1=1=-z_2$ and~$\bar{c}_1=\bar{c}_2=1$.  Here~$\|c_1\|$ is a weighted arc-length norm given by~\eqref{eq:weighted_arc-length}. Solid curves are stable branches with respect to PNP-CH dynamics, while dashed curve correspond to unstable solutions.  }
\label{fig:bifurcation_V1} 
\normalsize
\end{figure}

\begin{figure}[ht!]\begin{center}
\scalebox{0.9}{\includegraphics{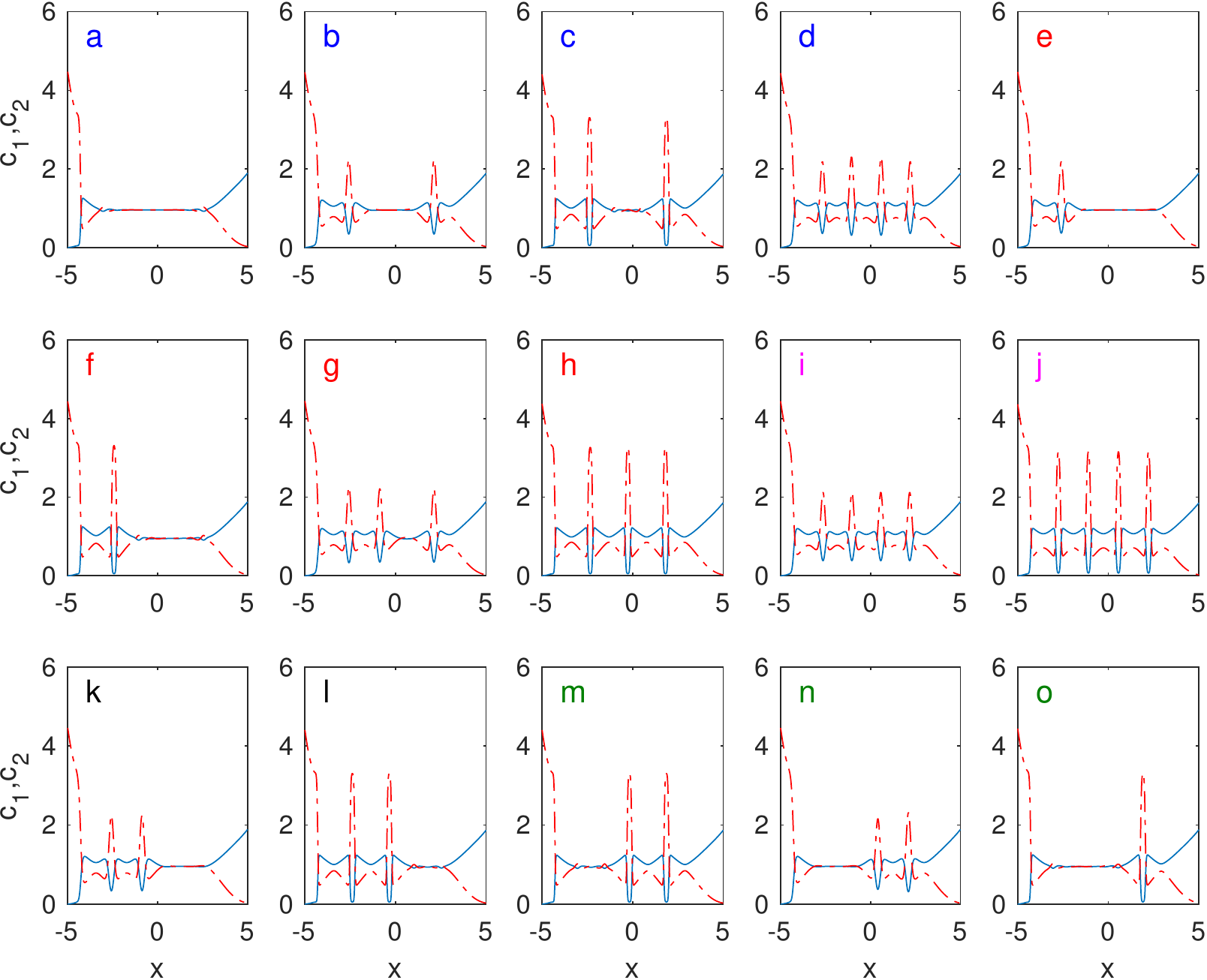}}
\end{center}
\caption{Charge concentration profiles~$c_1(x)$ ({\color{blue} solid}) and~$c_2(x)$ ({\color{red} dash-dots}) corresponding to points a-o of  Figure~\ref{fig:bifurcation_V1}. 
}
\label{fig:bifurcation_V1_solutions} 
\normalsize
\end{figure}
\section{Numerical methods}\label{sec:numerics}
The stationary solutions of the PNP-steric equation, presented in Sections~\ref{sec:convex} and~\ref{sec:nonConvex}, were computed by solving the initial value problem~(\ref{eq:StericPB_inverse},\ref{eq:IC}) using Matlab's ode45.  As noted, this is the reason many of the boundary-value problem parameters in Figures~\ref{fig:convexCase}-\ref{fig:crossD} are not round numbers.

The numerical continuation study of the stationary solutions of the PNP-CH system~(\ref{eq:stericPNP-CH},\ref{eq:genPNP_nondimensional_BC},\ref{eq:highOrderBC}) in a finite domain~$[-L,L]$ was conducted using pde2path 2.0~\cite{dohnal2014pde2path,uecker2014pde2path}.  
{ Resolving the bifurcation diagram for these system, however, is a demanding problem for standard continuation procedure.  Therefore, we utilize a combined method which integrates the numerical continuation procedure with simulations of the full PNP-CH system.} 
To test the stability of the stationary solutions computed by pde2path, we use LiveLink$^{\rm TM}$ for Matlab to solve the PNP-CH system, with initial condition attained from pde2path output, via a Comsol 5.2a finite element solver.  Whenever the stationary solution computed by pde2path is found to be unstable, the simulation of the PNP-CH system convergences to a new stable stationary solution.  If the  new stationary solution does not already reside on a known branch, it is used to find a new branch using numerical continuation. { Thus, there is a mutual flow of data between the continuation procedure and the PDE solver}. 

The pseudo-code is as follows.
\begin{enumerate}
\item Find stationary solution using numerical continuation, starting with a given (possibly approximate) stationary solution~${\bf u}_0^{\rm p2p}$.
\begin{enumerate}
\item At step i
\begin{enumerate}
\item Compute stationary solution~${\bf u}_i^{\rm p2p}$.
\item Solve the PNP-CH system with initial condition~${\bf u}(0,{\bf x})={\bf u}^{\rm p2p}_i$ until it reaches steady state at time~$T$.  Denote~${\bf u}^{\rm stationary}_i:={\bf u}(T,{\bf x})$.
\item The stationary solution~${\bf u}_i^{\rm p2p}$ is stable if
\[
\|{\bf u}^{\rm stationary}_i-{\bf u}_i^{\rm p2p}\|<{\rm TOL},
\]
and otherwise unstable.
\item If~${\bf u}_i^{\rm p2p}$ is unstable, and if~${\bf u}^{\rm stationary}_i$ does not reside on the same branch as the previously computed solution~${\bf u}^{\rm stationary}_{i-1}$, i.e.,
\[
\|{\bf u}_i^{\rm stationary}-{\bf u}_{i-1}^{\rm stationary}\|>TOL,
\]
then add~${\bf u}^{\rm stationary}_i$ to list~$\mathcal{L}$ of solutions on potentially new branches.
\end{enumerate}
\end{enumerate}
\item Save branch data~$\mathcal{B}_k=\{\sigma_i,\|{\bf u}_i^{\rm p2p}\|\}$
\item For each point in list~${\bf u}\in \mathcal{L}$ of solutions on potentially new branches, \begin{enumerate}\item verify that~${\bf u}$ does not reside on known branches~$\{\mathcal{B}_1,\mathcal{B}_2,\cdots,\mathcal{B}_k\}$ (i.e.,~$\|{\bf u}\|$ is at distance larger than TOL from branch curve).
\item Repeat stage 1, starting the numerical continuation from~${\bf u}$, i.e., setting~${\bf u}_0^{\rm p2p}= {\bf u}$
\end{enumerate}
\end{enumerate}
The combined method allows to map distinct solution branches more efficiently than considering solely the equations for the stationary states.  For example, we found that numerical continuation from point `a' retrieved the blue branch (associated with point `c') but not the red branch (associated with point `'e').  Therefore, mapping the three unstable branches emanating from point `a', see Figure~\ref{fig:bifurcation_V1}, using numerical continuation from this point is demanding.  Using the combined method, such issues are avoided.  For example, solving the PNP-CH system with initial conditions corresponding to solutions along the blue branch connecting points `a' and `b', gives rise to a stable stationary solution of the PNP-CH system on the red branch connecting points `e' and `f'.  The red branch is then traced back from this point to point `a' using continuation.
\section{Conclusions}\label{sec:discussion}
In this study, we have considered the existence and stability of stationary solutions of the PNP-steric system, and have shown that the { corresponding stationary} system gives rise to smooth multiple stationary solutions.  { The PNP-steric equation, however, turns out to be ill-posed at the parameter regimes where multiple solution arise.}  Following this finding,  we introduced a novel PNP-Cahn-Hilliard model which is valid at high concentrations.  We have shown that this model gives rise to multiple stationary solutions that are stable with respect to PNP-CH dynamics, and utilized bifurcation analysis and numerical continuation to map the various branches of stationary solutions.   In particular, we observe that the stationary solution profiles are essentially the same near the interfaces, and differ only by their behavior in the bulk, namely in regions where the electric field applied from the boundaries is screened.

 The PNP-steric model stems from the PNP-Lennard Jones (PNP-LJ) model which accounts for ion-ion steric interactions via the singular Lennard-Jones potential~\cite{eisenberg2010energy}.   The complicated structure of the PNP-LJ equations, however, gives rise to an highly demanding computational problem and to an intractable analytic problem.  Therefore, the PNP-steric model was derived using a leading order approximation of the Lennard-Jones potential~\cite{Horng:2012io}.  The fact that PNP-steric model fails at high concentrations, as shown in this work, implies that higher order approximations of the Lennard-Jones potential are required.  It is plausible that the second order approximation of the Lennard-Jones potential will give rise to gradient coefficient terms as in the PNP-CH model.  Derivation of a high-order PNP-steric model by utilizing more accurate approximations of the LJ potential is left for future work. 

The PNP-CH model differs from the broad family of generalized PNP models as it gives rise to bulk pattern formation.  From a physical point of view, emergence of periodic patterns or pattern formation, in general, does not seem natural in electrolyte systems - salt dissolved in a cup of water is not typically observed to spontaneously self-assemble into spatial nano-structures. Self-assembly of bulk nano-structure, however, is well documented in room-temperature ionic liquids, see~\cite{hayes2015structure} and references within.  Room-temperature ionic liquids are solvent-less mixtures of ions which remain liquid at room temperature.  In a sense, they provide a limiting example of the most concentrated electrolyte solutions.  Hence, sufficiently concentrated electrolyte solutions may also self-assemble into spatial nano-structures in the bulk~\cite{JPCL-Gavish-2017}.  

Electrolyte behavior near charged boundaries, aka the electric double layer structure, has been at the spotlight of electrolyte research due to  its importance within a wide range of areas including electrochemistry, biochemistry, physiology, and colloidal science.   
Recent works emphasize the significance of the coupling between bulk and interfacial nano-structures in such systems,~\cite{bier2017mean,gavish2017spatially,Yochelis2015,mezger2015solid}.
A study addressing the effect of bulk nano-structure on electrolyte behavior, and its coupling to the electric double layer structure will be presented elsewhere.


 Finally, one of the questions that motivated the study of bi-stability in modified PNP systems is whether bi-stability can explain gating phenomena in ionic channels.  This study does not address this question directly as the PNP-CH model with the no-flux boundary conditions considered in this study does not describe ion channels.  Nevertheless, the multiple solutions presented in this study are all associated with multiple bulk states,  and are not driven by boundary conditions, { see, e.g., Figure~\ref{fig:bifurcation_V1_solutions}}.  Thus, it is plausible that an appropriate PNP-CH model of ionic channels will give rise to similar multiple bulk states.  
 A systematic study of the stationary solutions of the PNP-CH model in an ion channel setting, and their corresponding current levels, will be presented elsewhere.

\section*{Acknowledgements}
This research was supported by the Technion VPR fund and from EU Marie--Curie CIG Grant 2018620.  The author thanks Hannes Uecker for valuable advice in the application of pde2path.  


\begin{thebibliography}{10}
\expandafter\ifx\csname url\endcsname\relax
  \def\url#1{\texttt{#1}}\fi
\expandafter\ifx\csname urlprefix\endcsname\relax\def\urlprefix{URL }\fi
\expandafter\ifx\csname href\endcsname\relax
  \def\href#1#2{#2} \def\path#1{#1}\fi

\bibitem{nernst1889elektromotorische}
W.~Nernst, Die elektromotorische wirksamkeit der jonen, Zeitschrift f{\"u}r
  physikalische Chemie 4~(1) (1889) 129--181.

\bibitem{planck1890ueber}
M.~Planck, Ueber die erregung von electricit{\"a}t und w{\"a}rme in
  electrolyten, Annalen der Physik 275~(2) (1890) 161--186.

\bibitem{gillespie2002coupling}
D.~Gillespie, W.~Nonner, R.~S. Eisenberg, Coupling poisson--nernst--planck and
  density functional theory to calculate ion flux, Journal of Physics:
  Condensed Matter 14~(46) (2002) 12129.

\bibitem{bikerman1942xxxix}
J.~Bikerman, Structure and capacity of electrical double layer, Philosophical
  Magazine 33 (1942) 384--397.

\bibitem{borukhov1997steric}
I.~Borukhov, D.~Andelman, H.~Orland, Steric effects in electrolytes: A modified
  poisson-boltzmann equation, Physical review letters 79~(3) (1997) 435.

\bibitem{kilic2007steric}
M.~S. Kilic, M.~Z. Bazant, A.~Ajdari, Steric effects in the dynamics of
  electrolytes at large applied voltages. ii. modified poisson-nernst-planck
  equations, Physical review E 75~(2) (2007) 021503.

\bibitem{ben2009beyond}
D.~Ben-Yaakov, D.~Andelman, D.~Harries, R.~Podgornik, Beyond standard
  poisson--boltzmann theory: ion-specific interactions in aqueous solutions,
  Journal of Physics: Condensed Matter 21~(42) (2009) 424106.

\bibitem{lopez2011poisson}
J.~J. L{\'o}pez-Garc{\'\i}a, J.~Horno, C.~Grosse, Poisson--boltzmann
  description of the electrical double layer including ion size effects,
  Langmuir 27~(23) (2011) 13970--13974.

\bibitem{stern1924theorie}
H.~O. Stern-Hamburg, Zur theorie{\textperiodcentered} der elektrolytischen
  doppelschicht., S. f. Electrochemie 30 (1924) 508.

\bibitem{di2004specific}
D.~di~Caprio, Z.~Borkowska, J.~Stafiej, Specific ionic interactions within a
  simple extension of the gouy--chapman theory including hard sphere effects,
  Journal of Electroanalytical Chemistry 572~(1) (2004) 51--59.

\bibitem{di2003simple}
D.~Di~Caprio, Z.~Borkowska, J.~Stafiej, Simple extension of the gouy-chapman
  theory including hard sphere effects.: Diffuse layer contribution to the
  differential capacity curves for the electrode electrolyte interface, Journal
  of Electroanalytical Chemistry 540 (2003) 17--23.

\bibitem{mansoori1971equilibrium}
G.~Mansoori, N.~Carnahan, K.~Starling, T.~Leland~Jr, Equilibrium thermodynamic
  properties of the mixture of hard spheres, The Journal of Chemical Physics 54
  (1971) 1523--1525.

\bibitem{Horng:2012io}
T.-L. Horng, T.-C. Lin, C.~Liu, B.~S. Eisenberg, {PNP Equations with Steric
  Effects: A Model of Ion Flow through Channels}, J. Phys. Chem. B 116 (2012)
  11422--11441.

\bibitem{gongadze2015asymmetric}
E.~Gongadze, A.~Igli{\v{c}}, Asymmetric size of ions and orientational ordering
  of water dipoles in electric double layer model-an analytical mean-field
  approach, Electrochimica Acta 178 (2015) 541--545.

\bibitem{ben2011ion}
D.~Ben-Yaakov, D.~Andelman, R.~Podgornik, D.~Harries, Ion-specific hydration
  effects: Extending the poisson-boltzmann theory, Current Opinion in Colloid
  \& Interface Science 16~(6) (2011) 542--550.

\bibitem{booth1951dielectric}
F.~Booth, The dielectric constant of water and the saturation effect, The
  Journal of Chemical Physics 19 (1951) 391--394.

\bibitem{hatlo2012electric}
M.~M. Hatlo, R.~Van~Roij, L.~Lue, The electric double layer at high surface
  potentials: The influence of excess ion polarizability, EPL (Europhysics
  Letters) 97~(2) (2012) 28010.

\bibitem{ben2011dielectric}
D.~Ben-Yaakov, D.~Andelman, R.~Podgornik, Dielectric decrement as a source of
  ion-specific effects, The Journal of chemical physics 134~(7) (2011) 074705.

\bibitem{psaltis2011comparing}
S.~Psaltis, T.~W. Farrell, Comparing charge transport predictions for a ternary
  electrolyte using the maxwell--stefan and nernst--planck equations, Journal
  of The Electrochemical Society 158~(1) (2011) A33--A42.

\bibitem{bazant2011double}
M.~Z. Bazant, B.~D. Storey, A.~A. Kornyshev, Double layer in ionic liquids:
  Overscreening versus crowding, Physical Review Letters 106~(4) (2011) 046102.

\bibitem{gavish2015systematic}
N.~Gavish, K.~Promislow, Systematic interpretation of differential capacitance
  data, Physical Review E 92~(1) (2015) 012321.

\bibitem{liu2014poisson}
J.-L. Liu, B.~Eisenberg, Poisson-nernst-planck-fermi theory for modeling
  biological ion channels, The Journal of chemical physics 141~(22) (2014)
  12B640\_1.

\bibitem{eisenberg2010energy}
B.~Eisenberg, Y.~Hyon, C.~Liu, Energy variational analysis of ions in water and
  channels: Field theory for primitive models of complex ionic fluids, The
  Journal of Chemical Physics 133~(10) (2010) 104104.

\bibitem{Bazant:2009fp}
M.~Z. Bazant, M.~S. Kilic, B.~D. Storey, A.~Ajdari, {Towards an understanding
  of induced-charge electrokinetics at large applied voltages in concentrated
  solutions}, Advances in Colloid and Interface Science 152 (2009) 48--88.

\bibitem{iglivc2015nanostructures}
A.~Igli{\v{c}}, D.~Drobne, V.~Kralj-Igli{\v{c}}, Nanostructures in Biological
  Systems: Theory and Applications, CRC Press, 2015.

\bibitem{lin2014new}
T.-C. Lin, B.~Eisenberg, A new approach to the lennard-jones potential and a
  new model: Pnp-steric equations, Communications in Mathematical Sciences
  12~(1) (2014) 149--173.

\bibitem{EJM:10386443}
N.~Gavish, K.~Promislow, On the structure of generalized {P}oisson--{B}oltzmann
  equations, European Journal of Applied Mathematics 27 (2016) 667--685.
\newblock \href {http://dx.doi.org/10.1017/S0956792515000613}
  {\path{doi:10.1017/S0956792515000613}}.

\bibitem{eisenberg1996atomic}
R.~Eisenberg, R.~Elber, Atomic biology, electrostatics and ionic channels, New
  developments and theoretical studies of proteins 7 (1996) 269--357.

\bibitem{kaufman2013multi}
I.~Kaufman, D.~Luchinsky, R.~Tindjong, P.~McClintock, R.~Eisenberg, Multi-ion
  conduction bands in a simple model of calcium ion channels, Physical biology
  10~(2) (2013) 026007.

\bibitem{lin2015multiple}
T.-C. Lin, B.~Eisenberg, Multiple solutions of steady-state poisson nernst
  planck equations with steric effects, Nonlinearity 28~(7) (2015) 2053, (first
  appeared as arXiv: 1407.8252 (2014)).

\bibitem{rubinstein1987multiple}
I.~Rubinstein, Multiple steady states in one-dimensional electrodiffusion with
  local electroneutrality, SIAM Journal on Applied Mathematics 47~(5) (1987)
  1076--1093.

\bibitem{liu2009one}
W.~Liu, One-dimensional steady-state poisson--nernst--planck systems for ion
  channels with multiple ion species, Journal of Differential Equations 246~(1)
  (2009) 428--451.

\bibitem{eisenberg2007poisson}
B.~Eisenberg, W.~Liu, Poisson--nernst--planck systems for ion channels with
  permanent charges, SIAM Journal on Mathematical Analysis 38~(6) (2007)
  1932--1966.

\bibitem{cahn1958free}
J.~W. Cahn, J.~E. Hilliard, Free energy of a nonuniform system. i. interfacial
  free energy, Journal of Chemical Physics 28 (1958) 258--267.

\bibitem{gavish2016theory}
N.~Gavish, A.~Yochelis, Theory of phase separation and polarization for pure
  ionic liquids, The journal of physical chemistry letters 7~(7) (2016)
  1121--1126.

\bibitem{Chun_Personal_Communication}
C.~Liu, Personal communications.

\bibitem{cox2001new}
S.~Cox, P.~Matthews, New instabilities in two-dimensional rotating convection
  and magnetoconvection, Physica D: Nonlinear Phenomena 149~(3) (2001)
  210--229.

\bibitem{golovin2003self}
A.~Golovin, S.~Davis, P.~Voorhees, Self-organization of quantum dots in
  epitaxially strained solid films, Physical Review E 68~(5) (2003) 056203.

\bibitem{schneider2013justification}
G.~Schneider, D.~Zimmermann, Justification of the ginzburg--landau
  approximation for an instability as it appears for marangoni convection,
  Mathematical Methods in the Applied Sciences 36~(9) (2013) 1003--1013.

\bibitem{dohnal2014pde2path}
T.~Dohnal, J.~Rademacher, H.~Uecker, D.~Wetzel, pde2path 2.0: Multi-parameter
  continuation and periodic domains, ENOC.

\bibitem{uecker2014pde2path}
H.~Uecker, D.~Wetzel, J.~D. Rademacher, pde2path-a matlab package for
  continuation and bifurcation in 2d elliptic systems, Numerical Mathematics:
  Theory, Methods and Applications 7~(01) (2014) 58--106.

\bibitem{hayes2015structure}
R.~Hayes, G.~G. Warr, R.~Atkin, Structure and nanostructure in ionic liquids,
  Chemical reviews 115 (2015) 6357--6426.

\bibitem{JPCL-Gavish-2017}
N.~Gavish, D.~Elad, A.~Yochelis,
  \href{http://dx.doi.org/10.1021/acs.jpclett.7b03048}{From solvent free to
  dilute electrolytes: Essential components for a continuum theory}, The
  Journal of Physical Chemistry Letters 0~(ja) (0) null, pMID: 29220577.
\newblock \href
  {http://arxiv.org/abs/http://dx.doi.org/10.1021/acs.jpclett.7b03048}
  {\path{arXiv:http://dx.doi.org/10.1021/acs.jpclett.7b03048}}, \href
  {http://dx.doi.org/10.1021/acs.jpclett.7b03048}
  {\path{doi:10.1021/acs.jpclett.7b03048}}.
\newline\urlprefix\url{http://dx.doi.org/10.1021/acs.jpclett.7b03048}

\bibitem{bier2017mean}
S.~Bier, N.~Gavish, H.~Uecker, A.~Yochelis, From bulk self-assembly to
  electrical diffuse layer in a continuum approach for ionic liquids: The
  impact of anion and cation size asymmetry, Physical Review E 95~(6) (2017)
  060201(R).

\bibitem{gavish2017spatially}
N.~Gavish, I.~Versano, A.~Yochelis, Spatially localized self-assembly driven by
  electrically charged phase separation, SIAM Journal on Applied Dynamical
  Systems 16~(4) (2017) 1946--1968.

\bibitem{Yochelis2015}
A.~Yochelis, M.~Singh, I.~Visoly-Fisher, Coupling bulk and near-electrode
  interfacial nanostructuring in ionic liquids, Chemistry of Materials 27
  (2015) 4169--4179.

\bibitem{mezger2015solid}
M.~Mezger, R.~Roth, H.~Schr{\"o}der, P.~Reichert, D.~Pontoni, H.~Reichert,
  Solid-liquid interfaces of ionic liquid solutions—interfacial layering and
  bulk correlations, Journal of Chemical Physics 142 (2015) 164707.

\end{thebibliography}

\end{document}